\documentclass[EJP]{ejpecp} 

\usepackage[utf8]{inputenc}

 \usepackage{enumerate}
\usepackage{enumitem}

\usepackage{hyperref}
\usepackage{bm}
\usepackage{youngtab}
\usepackage{tikz}

\SHORTTITLE{Statistics in random permutations}

\TITLE{
Asymptotic behavior of some statistics in Ewens random permutations}

\AUTHORS{%
    Valentin~Féray\footnote{
LaBRI, CNRS, Universit\'e Bordeaux 1, France.
\EMAIL{valentin.feray@labri.fr}}}


\AMSSUBJ{05A16; 05A05} 

\SUBMITTED{December 20, 2012} 
\ACCEPTED{August 7, 2013} 

\VOLUME{0}
\YEAR{2013}
\PAPERNUM{0}
\DOI{vVOL-PID}

\ABSTRACT{
    The purpose of this article is to present a general method to
    find limiting laws for some renormalized statistics on random permutations.
    The model of random permutations considered here is Ewens sampling
    model, which generalizes uniform random permutations.
    Under this model, we describe the asymptotic behavior of some
    statistics, including the number of occurrences of any dashed pattern.
    Our approach is based on the method of moments and relies on the following
    intuition: two events involving the images of different integers
    are {\em almost} independent.
}

\newcommand{\DD}{\bm{D}}
\newcommand{\C}{\mathcal{C}}
\newcommand{\CC}{\mathbb{C}}
\newcommand{\NN}{\mathbb{N}}
\newcommand{\PPP}{\mathcal{P}}

\newcommand{\Bex}{B^{\text{ex},N}}
\newcommand{\Bad}{B^{\text{ad},N}}
\newcommand{\Bc}{B^{\text{c},N}}
\newcommand{\ii}{\mathbf{i}}

\renewcommand{\ss}{\mathbf{s}}
\newcommand{\esper}{\mathbb{E}}
\newcommand{\PP}{\widetilde{\sigma}_{\ii,\ss}}
\DeclareMathOperator{\Conn}{CC}
\DeclareMathOperator{\Prob}{Prob}
\DeclareMathOperator{\Var}{Var}
\DeclareMathOperator{\Cov}{Cov}
\DeclareMathOperator{\rk}{rk}
\newcommand{\ev}{\text{ev}}
\newcommand{\odd}{\text{odd}}
\renewcommand{\forall}{\text{for all }}
\newcommand{\NNN}{\mathcal{N}}
\DeclareMathOperator{\Id}{Id}

\begin{document}



\section{Introduction}
\subsection{Background}
Permutations are one of the most classical objects in enumerative combinatorics.
Several statistics have been widely studied:
total number of cycles, number of cycles of a given length,
of descents, inversions, exceedances
or more recently, of occurrences of a given
(generalized) pattern... 
A classical question in enumerative combinatorics consists in computing the
(multivariate) generating series of permutations with respect to some of
these statistics.

A probabilistic point of view on the topic raises other questions.
Let us consider, for each $N$, a probability measure $\mu_N$ on permutations
of size $N$.
The simplest model of random permutations is of course the uniform
random permutations (for each $N$, $\mu_N$ is the uniform distribution on 
the symmetric group $S_N$).
A generalization of this model has been introduced by W.J. Ewens in the context of
population dynamics \cite{EwensDefinitionMeasure}.
It is defined by
\begin{equation}\label{DefEwens}
\mu_N(\{\sigma\}) = \frac{\theta^{\#(\sigma)}}
{\theta (\theta +1) \cdots (\theta + N-1)},
\end{equation}
where $\theta > 0$ is a fixed real parameter and $\#(\sigma)$ stands for the
number of cycles of the permutation $\sigma$.
Of course, when $\theta=1$, we recover the uniform distribution.
From now on, we will allow ourselves a small abuse of language and 
use the expression {\em Ewens random permutation} for a random permutation
distributed with Ewens measure.

Having chosen a sequence of probability disribution of $S_N$,
any statistic on permutations can be interpreted as a sequence of random variables 
$(X_N)_{N \geq 1}$.
The natural question is now: what is the asymptotic behavior 
(possibly after normalization) of $(X_N)_{N \geq 1}$?
The purpose of this article is to introduce a new general approach to
this family of problems, based on the method of moments.

We then use it to determine the second-order fluctuations of a large family
of statistics on permutations: occurrences of dashed patterns
(Theorem \ref{ThmDashedPatterns}).

Random permutations, either with uniform or Ewens distribution,
are well-studied objects.
Giving a complete list of references is impossible.
In Section \ref{SubsectComparaisonLitterature},
we compare our results with the literature.

\subsection{Motivating examples}\label{SubsectMotivations}
Let us begin by describing a few examples of results, covered by our method.
\bigskip

{\em Number of cycles of a given length $p$.}
Let $\Gamma^{(N)}_p$ be the random variable given by the number of cycles of
length $p$ in an Ewens random permutation $\sigma$ in $S_N$.
The asymptotic distribution of $\Gamma^{(N)}_p$ has been studied by V.L. Goncharov
\cite{GoncharovCycles} and V.F. Kolchin \cite{KolchinCyclesRandomPerm}
in the case of uniform measure and by G.A. Watterson
\cite[Theorem 5]{WattersonEwensMeasure}
in the framework of a general Ewens distribution
(see also \cite[Theorem 5.1]{LogarithmicCombinatorialStructures}).
\begin{theorem}[\cite{WattersonEwensMeasure}]
    Let $p$ be a positive integer.
    When $N$ tends to infinity, $\Gamma^{(N)}_p$ converges in distribution
    towards a Poisson law with parameter $\theta/{p}$.
    Moreover, the sequences of random variables
    $(\Gamma^{(N)}_{p'})_{N \geq 1}$ for $p'\leq p$
    are asymptotically independent.
    \label{ThmCycles}
\end{theorem}

%

{\em Exceedances.}
A (weak) exceedance of a permutation $\sigma$ in $S_N$ is an integer $i$
such that $\sigma(i) \geq i$.
Let $\Bex_i$ be the random variable defined by:
\[ \Bex_i(\sigma)=\begin{cases}
    0 & \text{ if } \sigma(i) < i;\\
    1 & \text{ if } \sigma(i) \geq i.
\end{cases} \]
When $\sigma$ is a Ewens random permutation,
this random variable is distributed according to
a Bernoulli law with parameter $\frac{i+\theta}{N+\theta-1}$
(see Lemma \ref{LemDistProdB}).

Let $x$ be a fixed real number in $[0;1]$
and $\sigma$ a permutation of size $N$.
When $Nx$ is an integer, we define
\[F^{(N)}_\sigma(x):= \frac{\sum_{i=1}^{N x} \Bex_i(\sigma)}{N}\]
and we extend the function $F^{(N)}_\sigma$ by linearity between
the points $i/N$ and $(i+1)/N$ (for $1 \leq i \leq N-1$).
In sections \ref{SubsectSSEP} and \ref{SubsectPermTableaux},
we explain why we are interested in this quantity:
it is related to a statistical physics model,
the {\em symmetric simple exclusion process} (SSEP),
and to permutation tableaux, some combinatorial objects
which have been intensively studied in the last few years.

We show the following.
\begin{theorem}
    Let $x$ be a real number between $0$ and $1$ and 
    $\sigma$ a random Ewens permutation of size $N$.
    Then, almost surely,
    \[ \lim_{N \to \infty} F^{(N)}_\sigma(x) = \frac{1 - (1-x)^2}{2}.\]
    Moreover, if we define the rescaled fluctuations
    \[Z^{(N)}_\sigma(x):=\sqrt{N} \left( F^{(N)}_\sigma(x) - \esper(F^{(N)}_\sigma(x)) \right),\]
    then, for any $x_1,\dots,x_r$, the vector $(Z^{(N)}_\sigma(x_1),\dots,Z^{(N)}_\sigma(x_r))$
    converges towards a Gaussian vector $(G(x_1),\dots,G(x_r))$
    with covariance matrix $(K(x_i,x_j))_{1 \leq i,j \leq r}$,
    for some explicit function $K$ (see Section~\ref{SubsectConv}).
    \label{ThmExc}
\end{theorem}
%
Although we have no interpretation for that,
let us note that the limit of $F^{(N)}_\sigma(x)$ is 
the cumulative distribution function of a $\beta$ variable with parameters $1$ and $2$.

With this formulation, Theorem~\ref{ThmExc} is new,
but the first part is quite easy 
while the second is a consequence of 
\cite[Appendix A]{DerridaLongRangeCorrelation}
(see Section \ref{SectExc}).
We also refer to an article of A. Barbour and S. Janson
\cite{BarbourJansonLimitPermTableaux},
where the case of the uniform measure is addressed with another method.
\bigskip

{\em Adjacencies.}
We consider here only uniform random permutations, that is the case $\theta=1$.
An adjacency of a permutation $\sigma$ in $S_N$ is an integer $i$ such that
$\sigma(i+1)=\sigma(i) \pm 1$.
As above, we introduce the random variable $\Bad_i$ which takes value $1$
if $i$ is an adjacency and $0$ otherwise.
Then $\Bad_i$ is distributed according to a Bernoulli law
with parameter $\frac{2}{N}$.
An easy computation shows that they are {\em not} independent.

We are interested in the total number of adjacencies in $\sigma$, that is
the random variable on $S_N$ defined by $A^{(N)}=\sum_{i=1}^{N-1} \Bad_i$.
\begin{theorem}[\cite{WolfowitzAdjacencies}]
    $A^{(N)}$ converges in distribution towards a Poisson variable with parameter $2$.
    \label{ThmAdj}
\end{theorem}
This result first appeared in papers of J. Wolfowitz and I. Kaplansky
\cite{WolfowitzAdjacencies, KaplanskyAdjacencies}
and was rediscovered recently in the context of genomics
(see \cite{AdjacencesPoisson} and
also \cite[Theorem 10]{BouvelEtAlFormeArbreDecomposition}).
\bigskip

In these three examples, the underlying variables behave asymptotically
as independent.
The main lemma of this paper is a precise statement of this {\em almost} independence,
that is an upper bound on joint cumulants.
This result allows us to give new proofs 
of the three results presented above in a uniform way.
Besides, our proofs follow the intuition that
events involving the image of different integers
are almost inedependent.

\subsection{The main lemma} \label{SubsectMainResult}
From now on, $N$ is a positive integer and $\sigma$ a random
Ewens permutation in $S_N$.
We shall use the standard notation $[N]$ 
for the set of the first $N$ positive integers.

If $i$ and $s$ are two integers in $[N]$, we consider the Bernoulli
variable $B^{(N)}_{i,s}$ which takes value $1$ if and only if $\sigma(i)=s$.
Despite its simple definition, this collection of events allows to reconstruct
the permutation and thus generates the full algebra of observables
(we call them {\em elementary events}).

For random variables $X_1,\dots,X_\ell$ on the same probability space
(with expectation denoted $\esper$),
we define their joint cumulant
\begin{equation}
    \kappa (X_1,\dots,X_\ell) = [t_1 \dots t_\ell] \ln 
    \bigg( \esper \big( \exp(t_1 X_1 + \dots + t_\ell X_\ell) \big) \bigg).
    \label{EqDefCumulant}
\end{equation}
As usual, $[t_1 \dots t_\ell] F$ stands for the coefficient of $t_1 \dots t_\ell$ 
in the series expansion of $F$ in positive powers of $t_1, \dots, t_\ell$.
Joint cumulants have been introduced by 
Leonov and Shiryaev \cite{LeonovShiryaevCumulants}.
For a summary of their most useful properties,
see \cite[Proposition~6.16]{JansonRandomGraphs}.

Our main lemma is a bound on joint cumulants of products of elementary events.
To state it, we introduce the following notations.
Consider two lists of positive integers of the same length
$\ii=(i_1,\dots,i_r)$ and $\ss=(s_1,\dots,s_r)$ and
define the graphs $G_1(\ii,\ss)$ and $G_2(\ii,\ss)$ as follows:
\begin{itemize}
  \item the vertex set of $G_1(\ii,\ss)$ is $[r]$ and 
  $j$ and $h$ are linked in $G_1(\ii,\ss)$ if and only if $i_j=i_h$ {\em and} $s_j=s_h$.
  \item the vertex set of $G_2(\ii,\ss)$ is also $[r]$ and
  $j$ and $h$ are linked in $G_2(\ii,\ss)$ if and only if $\{i_j,s_j\} \cap \{i_h,s_h\} \neq \emptyset$.
\end{itemize}
The connected components of a graph $G$ form a set partition of its vertex set that
we denote $\Conn(G)$.
Besides, if $\Pi$ is a set-partition $\Pi$,
we write $\#(\Pi)$ for its number of parts.
In particular, 
$\#(\Conn(G))$ is the number of connected components of $G$.
 
Finally, if $\pi_1$ and $\pi_2$, we denote $\pi_1 \vee \pi_2$
the finest partition which is coarser than $\pi_1$ and $\pi_2$
(here, {\em fine} and {\em coarse} refer to the refinement order; see
Section \ref{SubsecSetPartitions}).
\begin{theorem}[main lemma]
    Fix a positive integer $r$.
    There exists a constant $C_r$, depending on $r$,
    such that for any set partition $\tau=(\tau_1,\dots,\tau_\ell)$ of $[r]$,
    any $N \geq 1$ and lists $\ii=(i_1,\dots,i_r)$ and
    $\ss=(s_1,\dots,s_r)$ of integers in $[N]$, one has:
    \begin{equation}\label{EqMain}
        \left| \kappa\left( \prod_{j \in \tau_1} B^{(N)}_{i_j,s_j}, \dots,
        \prod_{j \in \tau_\ell} B^{(N)}_{i_j,s_j} \right) \right| 
        \leq C_r
        N^{-\#\big(\Conn(G_1(\ii,\ss))\big) - \#\big(\Conn(G_2(\ii,\ss)) \vee \tau \big) +1}.
    \end{equation}
    \label{ThMain}
\end{theorem}

Note that the integer $\#\big(\Conn(G_1(\ii,\ss))\big)$ is the number
of {\em different} pairs $(i_j,s_j)$.
The second quantity involved in the theorem $\#\big(\Conn(G_2(\ii,\ss)) \vee \tau \big)$
does not have a similar interpretation.
However, it admits an equivalent description.
Consider the graph $G'_2$,
obtained from $G_2(\ii,\ss)$ by merging
 vertices corresponding to elements in the same part of $\tau$.
 Then $\#\big(\Conn(G_2(\ii,\ss)) \vee \tau \big)$ is the number of 
 connected components of $G'_2$.

 As an example, let us consider the {\em distinct case},
 that is we assume that the entries in the lists
$\ii$ and $\ss$ are distinct.
We shall use the standard notation for falling factorials
$(x)_a=x\, (x-1)\, \cdots\, (x-a+1)$.
In this case, the expectation of a product of $B^{(N)}_{i,s}$ is simply 
$1/(N+\theta-1)_a$, where $a$ is the number of factors
(the case $\theta=1$ is obvious, while the general case is explained in
Lemma~\ref{LemDistProdB}).
Joint cumulants can be expressed in terms of joint moments
-- see \cite[Proposition 6.16 (vi)]{JansonRandomGraphs} --,
so the left-hand side of \eqref{EqMain} can be written
as an explicit rational function in $N$ of degree $-r$.
According to our main lemma, the sum has degree at most $-\ell - r +1$,
which means that many simplifications are happening
(they are not at all trivial to explain!).

This reflects the fact that the variables $B^{(N)}_{i_j,s_j}$ 
behave asymptotically as independent
(joint cumulants vanish when the set of variables can be split
into two mutually independent sets).

\begin{remark}
    It is worth noticing that our proof of the main lemma 
    goes through a very general criterion for a family of
    sequences of random variables to have small cumulants:
    see Lemma \ref{LemMultiplicativeCritetion}. 
    This may help to find a similar behaviour
    (that is random variables with small cumulants)
    in completely different contexts, see Section \ref{subsect:Future}.
\end{remark}

\subsection{Applications}
Recall that, if $Y^{(1)},\dots,Y^{(m)}$ are random variables such that the law of the 
$m$-tuple $(Y^{(1)},\dots,Y^{(m)})$ is entirely determined by its joint moments,
then the two following statements are equivalent 
(see \cite[Theorem 30.2]{BillingsleyProbMeasure}
for the analogous property in terms of moments).
\begin{itemize}
    \item For any $\ell$ and any list $i_1,\dots,i_\ell$ in $[m]$,
        \[\lim_{n \to \infty} \kappa \left(X_n^{(i_1)},\dots,X_n^{(i_\ell)} \right) 
        =\kappa \left( Y^{({i_1})},\dots,Y^{({i_\ell})} \right). \]
    \item The sequence of vectors $(X_n^{(1)},\dots,X_n^{(m)})$ converges in distribution
        towards the vector $(Y^{(1)},\dots,Y^{(m)})$.
\end{itemize}
As Gaussian and Poisson variables are determined by their moments (see
{\it e.g.} the criterion \cite[Theorem 30.1]{BillingsleyProbMeasure}),
cumulants can be used to prove convergence in distribution towards Gaussian or 
Poisson variables, such as the results of Section~\ref{SubsectMotivations}.

Theorem~\ref{ThMain} can be used to give new uniform proofs of
Theorems~\ref{ThmCycles}, \ref{ThmExc} and \ref{ThmAdj}.
Moreover, we get an extension of Theorem~\ref{ThmAdj}
to any value of the parameter $\theta$.


To give more evidence that our approach is quite general,
we study the number of occurrences of {\em dashed patterns}.
This notion has been introduced\footnote{In the paper of Babson and
 Steingrímsson, they are called generalized patterns. But, as some more general
{\em generalized patterns} have been introduced since (see next section),
we prefer to use {\em dashed patterns}.}
 in 2000 by E. Babson and E. Steingrímsson~\cite{SteingrimssonGeneralizedPatterns}.

\begin{definition}
A dashed pattern of size $p$ is the data of a permutation $\tau \in S_p$ and a subset $X$ 
of $[p-1]$.
An occurrence of the dashed pattern $(\tau,X)$ in a permutation $\sigma \in S_N$
is a list $i_1 < \dots < i_p$ such that:
\begin{itemize}
    \item for any $x \in X$, one has $i_{x+1}=i_x+1$.
    \item $\sigma(i_1),\dots,\sigma(i_p)$ is in the same relative order as
        $\tau(1),\dots,\tau(p)$.
\end{itemize}
The number of occurrences of the pattern $(\tau,X)$ will be denoted
$O_{\tau,X}^{(N)}(\sigma)$.
\label{DefDashedPaatterns}
\end{definition}

\begin{example}
    $O^{(N)}_{21,\emptyset}$ is the number of inversions,
    while $O^{(N)}_{21,\{1\}}$ is the number of descents.
    Many classical statistics on permutations can be written as the number
    of occurrences of a given dashed patten or
    as a linear combination of such statistics,
    see \cite[Section 2]{SteingrimssonGeneralizedPatterns}.
\end{example}

Thanks to our main lemma, we describe the second order asymptotics of
the number of occurrences of any given dashed pattern
in a random Ewens permutation.
\begin{theorem}
    Let $(\tau,X)$ be a dashed pattern of size $p$ (see definition
    \ref{DefDashedPaatterns})
    and $\sigma_N$ a sequence of random Ewens permutations.
    We denote $q=|X|$.
    Then, $\frac{O_{\tau,X}^{(N)}(\sigma_N)}{N^{p-q}}$,
    that is the renormalized number of occurrences of $(\tau,X)$,
     tends almost surely towards
    $\frac{1}{p!(p-q)!}$.
    Besides, one has the following central limit theorem:
    \[ Z^{(N)}_{(X,\tau)}:= \sqrt{N} \left( \frac{O_{\tau,X}^{(N)}}{N^{p-q}} - 
    \frac{1}{p! (p-q)!} \right) \to
    \NNN(0,V_{\tau,X}),\]
    where the arrow denotes a convergence in distribution and 
    $V_{\tau,X}$ is some nonnegative real number.
    \label{ThmDashedPatterns}
\end{theorem}
This theorem is proved in Section \ref{SubsectDashedPatterns} using Theorem \ref{ThMain}.

Unfortunately, we are not able to show in general that
the constant $V_{\tau,X}$ is positive 
($V_{\tau,X}=0$ would mean that we have not chosen the good normalization factor).
We formulate it as a conjecture.
\begin{conjecture}\label{ConjVariance}
    For any dashed pattern $(\tau,X)$, one has $V_{\tau,X} >0$.
\end{conjecture}

The following partial result has been proved by M. B\'ona 
\cite[Propositions 1 and 2]{BonaMonotonePatterns}
(M. B\'ona works with the uniform distribution,
but it should not be too hard to show that $V_{\tau,X}$ does not depend
on $\theta$).
\begin{proposition}
    For any $k \ge 1$, $\tau=\Id_k$ and $X=\emptyset$ or $X=[k-1]$,
    Conjecture~\ref{ConjVariance} holds true.
\end{proposition}

The proof relies on an expression of $V_{\tau,X}$
as a signed sum of products of binomial coefficients.
This expression can be extended to the general case
and we have checked by computer
that Conjecture~\ref{ConjVariance}
holds true for all patterns of size $8$ or less.


\subsection{Comparison with other methods}
\label{SubsectComparaisonLitterature}

There is a huge literature on random permutations.
While we will not make a comprehensive survey of the subject,
we shall try to present the existing methods and results related
to our paper.

Our Poisson convergence results have been obtained previously
by the moment method in the articles \cite{KaplanskyAdjacencies}
and \cite{WattersonEwensMeasure}.
Our cumulant approach is not really different from these proofs.
Yet, we have chosen to present these examples for two reasons:
\begin{itemize}
    \item first, it illustrates the fact that our approach
        can prove in a uniform ways
        convergence towards different distributions~;
    \item second, the combinatorics is simpler in the Poisson cases,
        so they serve as toy model to explain the general structure
        of the proofs.
\end{itemize}
Let us mention also the existence of a powerful method,
called the Stein-Chen method, that proves Poisson convergence,
together with precise bounds on total variation distances
-- see, {\em e.g.}, \cite[Chapter 4]{PoissonApproximation}.

Let us now consider our normal approximation results.
For uniform permutations, both are already known or could be obtained easily
with methods existing in the literature.
\begin{itemize}
    \item Theorem \ref{ThmExc} has been proved by A. Barbour and S. Janson
        \cite{BarbourJansonLimitPermTableaux},
        who established a functional version of a combinatorial central
        limit theorem from Hoeffding \cite{HoeffdingCombinatorialCLT}.
        This theorem deals with statistics of the form
        \[\sum_{1\le i,j\le N} a^{(N)}_{i,j} B^{(N)}_{i,j} \]
        where $A^{(N)}$ is a sequence of deterministic $N\times N$ matrices.
\item Theorem \ref{ThmDashedPatterns} has been proved for some particular
    patterns using dependency graphs and cumulants: see 
    \cite[Theorems 10 and 17]{BonaMonotonePatterns} and 
    \cite[Section 6]{HitczenkoJansonNormalityStatPermTab}.
    The case of a general pattern (under uniform distribution) can be handled
    with the same arguments.
\end{itemize}
These methods are very different one from each other
and none of them can be used to prove both results in a uniform way.
Note also that they only work in the uniform case.
Yet, going from the uniform model to a general Ewens distribution
should be doable using the {\em chinese restaurant process}
\cite[Example 2.4]{LogarithmicCombinatorialStructures}
(with this coupling, an Ewens random permutation 
differs from a uniform random permutation by $O(2|\theta-1| \log(n))$ values).

To conclude, while less powerful in the Poisson case,
our method has the advantage of providing a uniform proof for all these
results and to extend directly to a general Ewens distribution.

\subsection{Future work}\label{subsect:Future}
In addition to the conjecture above, we mention three directions for further
research on the topic.

It would be interesting to describe which permutation statistics
can be (asymptotically) studied with our approach.
This problem is discussed in Section \ref{SubsectLocalStat}.

Another direction consists in refining our convergence results
(speed of convergence, large deviations, local limit laws)
by following the same guideline.

Finally, it is natural to wonder if the method can be extended to other family of objects.
The extension to colored permutations should be straightforward.
A promising direction is the following:
consider a graph $G$ with vertex set $[n]$
and take some random subset $S$ of its vertices,
uniformly among all subsets of size $p$.
If $p$ grows linearly with $n$, then the events
``{\em $i$ lies in $S$}'' (for $1 \le i \le n$) have small joint cumulants
(this is easy to see with the material of Section \ref{SectProof}).

\subsection{Preliminaries: set partitions}
\label{SectPrelim}
\label{SubsecSetPartitions}
The combinatorics of set partitions is central in the theory of cumulants 
 and will be important in this article.
 We recall here some well-known facts about them.

A {\em set partition} of a set $S$ is a (non-ordered) family of non-empty disjoint
subsets of $S$ (called parts of the partition), whose union is $S$.

Denote $\PPP(S)$ the set of set partitions of a given set $S$.
Then $\PPP(S)$ may be endowed with a natural partial order:
the {\em refinement} order.
We say that $\pi$ is {\em finer} than $\pi'$ or $\pi'$ {\em coarser} than $\pi$
(and denote $\pi \leq \pi'$)
if every part of $\pi$ is included in a part of $\pi'$.

Endowed with this order, $\PPP(S)$ is a complete lattice, which means that
each family $F$ of set partitions admits a join (the finest set partition
which is coarser than all set partitions in $F$, denoted with $\vee$)
and a meet (the coarsest set partition
which is finer than all set partitions in $F$, denoted with $\wedge$).
In particular, there is a maximal element $\{S\}$ (the partition in only one
part) and a minimal element $\{ \{x\}, x \in S\}$ (the partition in singletons).

Moreover, this lattice is ranked:
the rank $\rk(\pi)$ of a set partition $\pi$ is $|S|-\#(\pi)$,
where $\#(\pi)$ denotes the number of parts of $\pi$.
The rank is compatible with the lattice structure in the following sense:
for any two set partitions $\pi$ and $\pi'$,
\begin{equation}\label{EqRkJoin}
\rk(\pi \vee \pi') \leq \rk(\pi) + \rk(\pi').
\end{equation}

Lastly, denote $\mu$ the M\"obius function of the partition lattice $\PPP(S)$.
In this paper, we only use evaluations of $\mu$ at pairs $(\pi,\{S\})$
(that is the second argument is the maximum element of $\PPP(S)$).
In this case, the value of the M\"obius function is given by:
\begin{equation}\label{EqValueMobius}
    \mu(\pi, \{S\})=(-1)^{\#(\pi)-1} (\# (\pi)-1)!.
\end{equation}

\subsection{Outline of the paper}
The paper is organized as follows.
Section \ref{SectProof} contains the proof of the main lemma.
Then, in Section \ref{SectGraph}, we give two easy lemmas on connected components
of graphs, which appear in all our applications.
The three last sections are devoted to the different applications:
Section \ref{SectCycles} for cycles, Section \ref{SectExc} for exceedances
and finally, Section \ref{SectPatterns} for generalized patterns (including 
adjacencies and dashed patterns).

\section{Proof of the main lemma}\label{SectProof}
This section is devoted to the proof of Theorem~\ref{ThMain}.
It is organized as follows.
First, we give a simple formula for the joint moments of the elementary events $(B_{i,s})$.
Second, we establish a general criterion based on joint moments
that implies that the corresponding variables have small joint cumulants.
Third, this criterion is used to prove Theorem~\ref{ThMain}
in the case of distinct indices.
The general case finally can be deduced from this particular case,
as shown in the last part of this section.

\subsection{Joint moments}\label{SubsectJointMoments}
The first step of the proof consists in
computing the joint moments of the family of random variables
$(B^{(N)}_{i,s})_{1 \leq i,s \leq N}$.

Note that $(B^{(N)}_{i,s})^2=B^{(N)}_{i,s}$,
while  $B^{(N)}_{i,s} B^{(N)}_{i,s'}=0$ if $s\neq s'$ and
$B^{(N)}_{i,s} B^{(N)}_{i',s}=0$ if $i\neq i'$.
Therefore, we can restrict ourselves to the computation of
the joint moment 
$\esper \left( B^{(N)}_{i_1,s_1} \cdots B^{(N)}_{i_r,s_r} \right) $,
in the case where $\ii=(i_1, \dots, i_r)$ and $\ss=(s_1, \dots, s_r)$ are two lists
of distinct indices (some entry of the list $\ii$
can be equal to an entry of $\ss$).

We see these two lists as a {\em partial permutation}
\[\widetilde{\sigma}_{\ii,\ss} = 
\left( \begin{array}{ccc} i_1 & \dots & i_r \\ s_1 & \dots &s_r \end{array}\right), \]
which sends $i_j$ to $s_j$.
The notion of cycles of a permutation can be naturally extended to partial permutations:
$(i_{j_1},\dots,i_{j_\gamma})$ is a cycle of the partial permutation
if $s_{j_1}=i_{j_2}$, $s_{j_2}=i_{j_3}$ and so on until $s_{j_\gamma}=i_{j_1}$.
Note that a partial permutation does not necessarily have cycles.
The number of cycles of $\widetilde{\sigma}_{\ii,\ss}$
is denoted $\#(\widetilde{\sigma}_{\ii,\ss})$.

The computation of $\esper \left( B^{(N)}_{i_1,s_1} \cdots B^{(N)}_{i_r,s_r} \right) $
relies on two important properties of the Ewens measure.
First, it is conjugacy-invariant.
Second, a random sampling can be obtained inductively by the following procedure
(see, {\it e.g.} \cite[Example 2.19]{LogarithmicCombinatorialStructures}).

Suppose that we have a random Ewens permutation $\sigma$ of size $N-1$.
Write it as a product of cycles and apply the following random transformation.
\begin{itemize}
\item With probability $\theta/(N+\theta-1)$, add $N$ as a fixed point.
More precisely, $\sigma'$ is defined by:
\[\begin{cases}\sigma'(i)=\sigma(i) & \text{for } i <N;\\
\sigma'(N)=N.
\end{cases}\]
\item For each $j$, with probability $1/(N+\theta-1)$, add $N$ just before $j$ in its cycle.
More precisely, $\sigma'$ is defined by:
\[\begin{cases}\sigma'(i)=\sigma(i) & \text{for } i  \neq \sigma^{-1}(j),N;\\
\sigma'(N)=j;\\
\sigma'(\sigma^{-1}(j))=N.
\end{cases}\]
\end{itemize}
Then $\sigma'$ is a random Ewens permutation of size $N$.
Iterating this, one obtains a linear time and space algorithm
to pick a random Ewens permutation.

Let us come back now to the computation of joint moments.
The following lemma may be known, 
but the author has not been able to find it in the literature.
\begin{lemma}\label{LemDistProdB}
Let $\sigma$ be a random Ewens permutation.
Then one has
\[ \esper \left( B^{(N)}_{i_1,s_1} \cdots B^{(N)}_{i_r,s_r} \right) 
= \frac{\theta^{\#(\widetilde{\sigma}_{\ii,\ss})} }{
(N+\theta-1) \dots (N+\theta-r)}.\]
\end{lemma}
For example, the parameters of the Bernoulli variables $B^{(N)}_{i,s}$ are given by
\[\esper(B^{(N)}_{i,s}) = \begin{cases}
\frac{\theta}{N+\theta-1} &\text{ if } i=s ; \\
\frac{1}{N+\theta-1} &\text{ if } i \neq s.
\end{cases}\]
 \begin{proof}
As Ewens measure is constant on conjugacy classes of $S_N$, one can assume without
loss of generality that $i_1=N-r+1$, $i_2=N-r+2$, \dots, $i_r=N$.
Then permutations of $S_N$ with $\sigma(i_j)=s_j$ are obtained in the previous algorithm as follows:
\begin{itemize}
\item Choose any permutation in $S_{N-r}$.
\item For $1 \leq j \leq r$, add $i_j$ in the place given by the following rule:
if $s_j<i_j$, add $i_j$ just before $s_j$ in its cycle.
Otherwise, look at $\PP(i_j)$, $\PP^2(i_j)$ and so on until you find an element smaller than $i_j$
and place $i_j$ before it.
If there is no such element, then $i_j$ is a minimum of a cycle of $\PP$.
In this case, put it in a new cycle.
\end{itemize}
It is easy to check with the description of the construction of a permutation under Ewens measure
that these choices of places happen with a probability
\[\frac{\theta^{\#(\widetilde{\sigma}_{\ii,\ss})} }
{(N+\theta-1) \dots (N-r+\theta)}. \qedhere\]
 \end{proof}

\subsection{A general criterion for small cumulants}
Let $A^{(N)}_1$,\dots,$A^{(N)}_\ell$ be $\ell$ sequences of random variables.
We introduce the following notation for joint moments and cumulants of subsets
of these variables:
for a subset $\Delta=\{j_1,\dots,j_h\}$ of $[\ell]$, we write
\[M_{A,\Delta}^{(N)}=\esper\left(  A^{(N)}_{j_1} \dots A^{(N)}_{j_h} \right),
\quad \kappa_{A,\Delta}^{(N)} = \kappa\left(  A^{(N)}_{j_1}, \dots, A^{(N)}_{j_h} \right).\]
We also introduce the auxiliary quantity $U_{A,\Delta}^{(N)}$ implicitly defined by
the property: 
for any subset $\Delta \subseteq [\ell]$,
\[ \prod_{\delta \subseteq \Delta} U_{A,\delta}^{(N)} = M_{A,\Delta}^{(N)}.\]
Using M\"obius inversion on the boolean lattice, we have explicitly:
for any subset $\Delta \subseteq [\ell]$,
\[ U_{A,\Delta}^{(N)}= \prod_{\delta \subseteq \Delta}                       
                \left( M_{A,\delta}^{(N)} \right)^{(-1)^{|\delta|}}\]

\begin{lemma}
    Let $A^{(N)}_1, \dots, A^{(N)}_\ell$ be a list of
    sequences of random variables
    with normalized expectations, that is,
    for any $N$ and $j$, $\esper(A^{(N)}_j) = 1$.
    Then the following statements are equivalent:
    \begin{enumerate}[label=\Roman*]
        \item \hspace{-.2cm}. \label{ItemQuasiFactorization}
            {\em Quasi-factorization property:}
            for any subset $\Delta \subseteq [\ell]$ of size at least $2$, one has
            \begin{equation}\label{EqQuasiFactorization}
                U_{A,\Delta}^{(N)}
                = 1 + O(N^{-|\Delta|+1}) ;
            \end{equation}
        \item \hspace{-.2cm}. \label{ItemSmallCumulants}
            {\em Small cumulant property:}
            for any subset $\Delta \subseteq [\ell]$ of size at least $2$, one has
            \begin{equation}\label{EqSmallCumulants}
                \kappa_{A,\Delta}^{(N)} = O(N^{-|\Delta|+1}). 
            \end{equation}
    \end{enumerate}
    \label{LemMultiplicativeCritetion}
\end{lemma}
\begin{proof}
    Let us consider the implication {\it \ref{ItemQuasiFactorization}} $\Rightarrow$ 
    {\it \ref{ItemSmallCumulants}}.
    We denote $T^{(N)}_\Delta=U_{A,\Delta}^{(N)}-1$ 
    and assume that $T^{(N)}_\Delta=O(N^{-|\Delta|+1})$ for any $\Delta \subseteq [\ell]$
    of size at least $2$.
    The goal is to prove that $ \kappa_{A,[\ell]}^{(N)}= O(N^{-\ell+1})$.
    Indeed, this corresponds to the case $\Delta = [\ell]$ of {\it \ref{ItemSmallCumulants}},
    but the same proof will work for any $\Delta \subseteq [\ell]$.

    Recall the well-known relation between joint moments and cumulants 
    \cite[Proposition 6.16 (vi)]{JansonRandomGraphs}:
    \begin{equation}\label{EqMoment2Cumulant}
    \kappa_{A,[\ell]}^{(N)} = \sum_{\pi \in \PPP([\ell])} \mu(\pi, \{[\ell]\})
        \prod_{C \in \pi} M_{A,C}^{(N)}.
    \end{equation}
    But joint moments can be expressed in terms of $T$:
    \[ M_{A,C}^{(N)} =
    \prod_{\Delta \subseteq C \atop |\Delta| \geq 2} (1 + T^{(N)}_\Delta) 
    = \sum_{\Delta_1,\dots,\Delta_m} T^{(N)}_{\Delta_1} \dots T^{(N)}_{\Delta_m},
    \]
    where the sum runs over all finite lists of distinct
    (but not necessarily disjoint) subsets of $C$
    of size at least $2$ (in particular, the length $m$ of the list is not fixed).
    When we multiply this over all blocks $C$ of a set partition $\pi$,
    we obtain the sum of $T^{(N)}_{\Delta_1} \dots T^{(N)}_{\Delta_m}$ over all lists 
    of distinct subsets of $[\ell]$ of size at least $2$
    such that each $\Delta_i$ is contained in a block of $\pi$.
    In other terms, for each $i \in [m]$, $\pi$ must be coarser
    than the partition $\Pi(\Delta_i)$,
    which, by definition, has $\Delta_i$ and singletons as blocks.
    Finally,
    \begin{equation}\label{EqCumulantsSumSets}
        \kappa_{A,[\ell]}^{(N)} = \sum_{\Delta_1,\dots,\Delta_m \atop
    \text{distinct}} T^{(N)}_{\Delta_1} \dots T^{(N)}_{\Delta_m} 
    \left( \sum_{\pi \in \PPP([\ell]) \atop \forall i, \ \pi \geq \Pi(\Delta_i)}
    \mu(\pi, \{[\ell]\}) \right).
\end{equation}
The condition on $\pi$ can be rewritten as
\[\pi \geq \Pi(\Delta_1) \vee \dots \vee  \Pi(\Delta_m).\]
Hence, by definition of the M\"obius function, the sum in the parenthesis
is equal to $0$, unless $\Pi(\Delta_1) \vee \dots \vee  \Pi(\Delta_m) = \{[\ell]\}$
(in other terms, unless the hypergraph with edges $(\Delta_i)_{1\le i \le m}$ is
connected).
On the one hand, by Equation~\eqref{EqRkJoin}, it may happen only if:
\[ \sum_{i=1}^m \rk\big(\Pi(\Delta_i)\big) = \sum_{i=1}^m (|\Delta_i|-1) \geq \rk([\ell]) = \ell - 1.\]
On the other hand, one has
\[T^{(N)}_{\Delta_1} \dots T^{(N)}_{\Delta_m} = O \left( N^{-\sum_{i=1}^m (|\Delta_i|-1)} \right).\]
Hence only summands of order of magnitude $N^{-\ell + 1}$ or less survive and one has
\[\kappa_{A,[\ell]}^{(N)} = O(N^{-\ell+1}) \]
which is exactly what we wanted to prove.
\bigskip

Let us now consider the converse statement.
We proceed by induction on $\ell$ and we assume that, 
for all $\ell'$ smaller than a given $\ell \geq 2$, 
the theorem holds.

Consider some sequences of random variables $A^{(N)}_1$, \dots, $A^{(N)}_\ell$
such that {\it \ref{ItemSmallCumulants}} holds.
By induction hypothesis, one gets immediately that 
\[ \forall \Delta \subsetneq [\ell],       
      U_{A,\Delta}^{(N)}=1+O(N^{-|\Delta|+1}). \]
Note that an immediate induction shows that the joint moment fulfills
\[ \forall \Delta \subsetneq [\ell],\ M_{A,\Delta}^{(N)}=O(1) \text{ and }
(M_{A,\Delta}^{(N)})^{-1}=O(1).\]
It remains to prove that 
\[ U_{A,[\ell]}^{(N)} = \prod_{\Delta \subseteq [\ell]}       
      (M_{A,\Delta}^{(N)})^{(-1)^{|\Delta|}}=1+O(N^{-\ell+1}).\]
Thanks to the estimate above for joint moments, this can be rewritten as
\begin{equation}\label{EqRewriteQuasiFact}
    M^{(N)}_{A,[\ell]} = \prod_{\Delta \subsetneq [\ell]} 
      (M_{A,\Delta}^{(N)})^{(-1)^{\ell-1-|\Delta|}}
      + O(N^{-\ell+1}).
  \end{equation}
Consider $\ell$ sequences of random variables $B^{(N)}_1$,\dots, $B^{(N)}_\ell$ such that,
for $\Delta \subsetneq [\ell]$,
the equality $M_{B,\Delta}^{(N)}=M_{A,\Delta}^{(N)}$ holds,
and such that 
Equation~\eqref{EqRewriteQuasiFact} is fulfilled when $A$ is replaced by $B$ 
(the reader may wonder whether such a family $B$ exists;
let us temporarily ignore this problem, which will be addressed in Remark \ref{RemUmbral}).
By definition, the family $B$ of sequences of random variables fulfills condition 
{\it \ref{ItemQuasiFactorization}} of the theorem and, hence, using the first
part of the proof, has also property {\it \ref{ItemSmallCumulants}}. 
In particular:
\[\kappa_{B,[\ell]}^{(N)} = O(N^{-\ell+1}).\]
But, by hypothesis,
\[\kappa_{A,[\ell]}^{(N)} = O(N^{-\ell+1}).\]
As $A$ and $B$ have the same joint moments, except for $M^{(N)}_{A,[\ell]}$
and $M^{(N)}_{B,[\ell]}$, this implies that
\[M_{A,[\ell]}^{(N)} - M_{B,[\ell]}^{(N)} =
\kappa_{A,[\ell]}^{(N)} - \kappa_{B,[\ell]}^{(N)} =
O(N^{-\ell+1}).\]
But the family $B$ fulfills Equation~\eqref{EqRewriteQuasiFact}
and, hence, so does family $A$.
\end{proof}

\begin{remark}
    \label{RemUmbral}
    Let $\ell$ be a fixed integer and $I$ a finite subset of $(\NN_{>0})^\ell$.
    Then, for any list $(m_\ii)_{\ii \in I}$ of numbers, one can find some
    {\em complex-valued} random variables $X_1,\dots,X_\ell$ so that
    \[\esper(X_1^{i_1}\dots X_\ell^{i_\ell}) = m_{i_1,\dots,i_\ell}.\]
    Indeed, one can look for a solution where
    $X_1$ is uniform on a finite set $\{z_1,\dots,z_T\}$
    and $X_j=X_1^{d^{j-1}}$, where $d$ is an integer bigger than all coordinates
    of all vectors in $I$.
    Then the quantities
    \[\{T \cdot \esper(X_1^{i_1}\dots X_\ell^{i_\ell}), \ii \in I\}\]
    correspond to different power sums of $z_1,\dots,z_T$.
    Thus we have to find a set $\{z_1,\dots,z_T\}$ of complex numbers
    with specified power sums up to degree $d^j$.
    This exists as soon as $T \geq d^j$, because $\CC$ is algebraically closed.
    In particular, the family $B$ considered in the proof above exists.

    However, we do not really need that this family exists.
    Indeed, during the whole proof, we are doing manipulations on the sequences
    of moments and cumulants using only the relations between them
    (equation \eqref{EqMoment2Cumulant}).
    We never consider the underlying random variables.
    Therefore, everything could be done even if the random variables 
    did not exist, as it is often done in umbral calculus
    \cite{RotaUmbral}.
\end{remark}

\subsection{Case of distinct indices}
\label{SubsecProofDistinct}
Recall that, in the statement of Theorem~\ref{ThMain},
we fix a set-partition $\tau$ and two lists $\ii$ and $\ss$ and we want to bound
the quantity
\[ \left| \kappa\left( \prod_{j \in \tau_1} B^{(N)}_{i_j,s_j}, \dots,\prod_{j \in \tau_\ell} B^{(N)}_{i_j,s_j} \right) \right|.\]
We first consider the case where all entries in the sequences
$\ii$ and $\ss$ are distinct.
To be in the situation of Lemma \ref{LemMultiplicativeCritetion},
we set, for $h \in [\ell]$ and $N \geq 1$:
\[ A^{(N)}_h =(N+\theta-1)_{a_j} \prod_{j \in \tau_h} B^{(N)}_{i_j,s_j},\]
where $a_j=|\tau_j|$.
The normalization factor has been chosen so that $\esper(A^{(N)}_h)=1$.
Hence, we will be able to apply Lemma \ref{LemDistProdB}.

Let us prove that $A^{(N)}_1, \dots, A^{(N)}_\ell$ fulfills property
{\it \ref{ItemQuasiFactorization}} of this lemma.
Of course, the case $\Delta=[\ell]$ is generic.
Thanks to Lemma \ref{LemDistProdB},
the joint moments of the family $A$ have in this case an explicit expression:
for $\delta \subseteq [\ell]$,
\[M_{A,\delta}^{(N)}=\frac{\displaystyle \prod_{j \in \delta}
(N+\theta-1)_{a_j}}{\displaystyle (N+\theta-1)_{\sum_{j \in \delta} a_j}}.\]
Therefore, we have to prove that the quantity
\[Q_{a_1,\dots,a_\ell}:=\prod_{\delta \subseteq [\ell] \atop |\delta| \geq 2} (M^{(N)}_{A,\delta})^{(-1)^{|\delta|}} = 
\prod_{\delta \subseteq [\ell]} \left((N+\theta-1)_{\sum_{j \in \delta} a_j}\right)^{(-1)^{|\delta|+1}}\]
is $1+O(N^{-\ell+1})$. \medskip

We proceed by induction over $a_\ell$.
If $a_\ell=0$, for any $\delta \subseteq [\ell -1]$,
the factors corresponding to $\delta$ and $\delta \sqcup \{\ell\}$
cancel each other.
Thus $Q_{a_1,\dots,a_{\ell-1},0}=1$ and the statement holds.

If $a_\ell>0$, the quantity $Q_{a_1,\dots,a_\ell}$ can be written as
\[Q_{a_1,\dots,a_\ell} = Q_{a_1,\dots,a_\ell-1} \cdot 
\prod_{\delta \subseteq [\ell] \atop \ell \in \delta}
\left( N+\theta -1- \sum_{j \in \delta} a_j \right)^{(-1)^{|\delta|+1}}.\]
Setting $X=N+\theta-1-a_{\ell}$, the second factor becomes
\[R_{a_1,\dots,a_{\ell-1}}(X):=\prod_{\delta \subseteq [\ell-1]} \left( X - \sum_{j \in \delta} a_j \right)^{(-1)^{|\delta|}}.\]
We will prove below (Lemma \ref{LemTechCumulants}) that 
$R_{a_1,\dots,a_{\ell-1}}(X) = 1+ O(X^{-\ell+1})$,
when $X$ goes to infinity.
Besides, the induction hypothesis implies that 
$Q_{a_1,\dots,a_\ell-1}= 1+ O(N^{-\ell+1})$
and hence 
\[Q_{a_1,\dots,a_\ell} = 1+ O(N^{-\ell+1}) \]
Using the terminology of lemma \ref{LemMultiplicativeCritetion},
it means that the list $A^{(N)}_1, \dots, A^{(N)}_\ell$ of sequences of
random variables has the quasi-factorisation property.
Thus it also has the small cumulant property and in particular
\[ \kappa(A^{(N)}_1,\dots, A^{(N)}_\ell) = O(N^{-\ell+1}). \]
Using the definition of the $A^{(N)}_h$, this can be rewritten:
\[ \kappa\left( \prod_{j \in \tau_1} B^{(N)}_{i_j,s_j}, \dots,
\prod_{j \in \tau_\ell} B^{(N)}_{i_j,s_j} \right)
= O(N^{-r-\ell+1}), \]
which is Theorem \ref{ThMain} in the case of distinct indices. \qed \bigskip

Here is the technical lemma that we left behind in the proof.
\begin{lemma}\label{LemTechCumulants}
For any positive integers $a_1,\dots,a_{\ell-1}$,
\[\prod_{\delta \subseteq [\ell-1]} \left( X - \sum_{j \in \delta} a_j \right)^{(-1)^{|\delta|}}= 1+ O(X^{-\ell+1}), \]
when $X$ is a positive number going to infinity.
\end{lemma}
\begin{proof}
Define $R_\ev$ (resp. $R_\odd$) as
\[\prod_\delta \left( X - \sum_{j \in \delta} a_j \right),\]
where the product runs over subsets of $[\ell-1]$ of even (resp. odd) size.
Expanding the product, one gets
\[R_\ev = \sum_{m \geq 0} \ \ 
\sum_{\delta_1,\dots,\delta_m} \ \ \sum_{j_1\in \delta_1,\dots,j_m \in \delta_m}  (-1)^m  
a_{j_1} \dots a_{j_m} X^{2^{\ell-2} - m}.\]
The index set of the second summation symbol is the set of lists of $m$ distinct
(but not necessarily disjoint) subsets of $[\ell-1]$ of even size.
Of course, a similar formula with subsets of odd size holds for $R_\odd$.

Let us fix an integer $m<\ell-1$ and a list $j_1,\dots,j_m$.
Denote $j_0$ the smallest integer in $[\ell-1]$ different from $j_1,\dots,j_m$
(as $m<\ell-1$, such an integer necessarily exists).
Then one has a bijection:
\[ \begin{array}{rcl}
\left\{ \begin{array}{c}
    \text{lists of subsets}\\
\delta_1,\dots,\delta_m \text{ of even size such}\\
\text{that, }\forall\ h\leq m, j_h \in \delta_h
\end{array}\right\}
& \to &
\left\{ \begin{array}{c}
    \text{lists of subsets}\\
\delta_1,\dots,\delta_m \text{ of odd size such}\\
\text{that, }\forall\ h\leq m, j_h \in \delta_h
\end{array}\right\}\\
(\delta_1,\dots,\delta_m)
&\mapsto&
(\delta_1 \nabla \{j_0\},\dots,\delta_m \nabla \{j_0\}),
\end{array}\]
where $\nabla$ is the symmetric difference operator.
This bijection implies that the summand $ (-1)^m  
a_{j_1} \dots a_{j_m} X^{2^{\ell-2} - m}$ appears as many times
in $R_\ev$ as in $R_\odd$.
Finally, all terms corresponding to values of $m$ smaller than $\ell-1$
cancel in the difference
$R_\ev - R_\odd$ and one has
\[R_\ev - R_\odd = O\big(X^{2^{\ell-2} - \ell +1}\big). \qedhere\]
\end{proof}

\begin{remark}
    Thanks to a result of Leonov and Shiryaev that expresses cumulants
    of products of random variables as product of cumulants
    (see \cite{LeonovShiryaevCumulants} or \cite[Theorem 4.4]{Sniady2006fluctuations}),
    it would have been enough to prove our result for $a_1=\dots=a_\ell=1$.
    But, as our proof uses an induction on $a_\ell$, 
    we have not made this choice.
\end{remark}

\begin{remark}
We would like to point out the fact that our result is
closely related to a result of P. \'Sniady.
Indeed, thanks to our multiplicative criterion to have small cumulants,
the computation in this section is equivalent to
Lemma 4.8 of paper \cite{Sniady2006fluctuations}.
However, \'Sniady's proof relies on a non trivial theory of cumulants of
observables of Young diagrams.
Therefore, it seems to us that it is worth giving an alternative argument.
%
\end{remark}

\subsection{General case}
Let $A^{(N)}_1$, \dots, $A^{(N)}_\ell$ be some sequences of random variables.
We introduce some {\em truncated cumulants}:
if $\pi_0$, $\pi_1$, $\pi_2$ and so on, are set partitions of $[\ell]$, we set
\begin{align*}
    k^{(N)}_{A}(\pi_0)&=\sum_{\pi \in \PPP([\ell]) \atop \pi \geq \pi_0}
    \mu(\pi, \{[\ell]\})
    \prod_{C \in \pi} M_{A,C}^{(N)}\\
    k^{(N)}_{A}(\pi_0;\pi_1,\pi_2,\dots)&=\sum_{\pi \in \PPP([\ell]) \atop 
    {\pi \geq \pi_0 \atop \pi \nleq \pi_1,\pi_2,\dots}} \mu(\pi, \{[\ell]\})
    \prod_{C \in \pi} M_{A,C}^{(N)}
\end{align*}

In the context of Lemma \ref{LemMultiplicativeCritetion}, it is also possible
to bound the truncated cumulants.
\begin{lemma}
    Let $A^{(N)}_1$,\dots,$A^{(N)}_\ell$ be some sequences
    of random variables as in Lemma \ref{LemMultiplicativeCritetion},
    fulfilling property {\it \ref{ItemQuasiFactorization}}
    (or equivalently property {\it \ref{ItemSmallCumulants}}).
    \begin{itemize}
        \item If $\pi_0$ is a set partition of $[\ell]$,
            \[k^{(N)}_{A}(\pi_0) = O(N^{-\#(\pi_0)+1}).\]
        \item More generally, if $\pi_0;\pi_1,\pi_2,\dots$ are set partitions
            of $[\ell]$,
            \[k^{(N)}_{{A}}(\pi_0;\pi_1,\pi_2,\dots)
    = O(N^{-\#(\pi_0 \vee \pi_1 \vee \pi_2 \dots)+1}).\]
    \end{itemize}
    \label{LemCumulantsTronques}
\end{lemma}

\begin{proof}
    For the first statement, the proof is similar to the one of
    {\it \ref{ItemQuasiFactorization}} $\Rightarrow$ {\it \ref{ItemSmallCumulants}}
    of Lemma \ref{LemMultiplicativeCritetion}.
    One can write an analogue of equation \eqref{EqCumulantsSumSets}:
    \[
        k^{(N)}_{{A}}(\pi_0)
       = \sum_{\Delta_1,\dots,\Delta_m \atop
    \text{distinct}} T^{(N)}_{\Delta_1} \dots T^{(N)}_{\Delta_m} 
    \left( \sum_{\pi \in \PPP([\ell]) \atop 
    \pi \geq (\pi_0 \vee \pi(\Delta_1) \vee \dots)}
    \mu(\pi, \{[\ell]\}) \right).
    \]
    The same argument as above says that only terms corresponding to lists
    such that $\pi_0 \vee \pi(\Delta_1) \vee \dots={[\ell]}$ survive.
    Such lists fulfill
    \[\sum_{i=1}^m |\Delta_i|-1 \geq \rk([\ell])-\rk(\pi_0)=\#(\pi_0)-1.\]
    The first item of the Lemma follows because, by hypothesis,
    \[T^{(N)}_{\Delta_1} \dots T^{(N)}_{\Delta_m} = O(N^{-\sum_i (|\Delta_i|-1)}).\]
    For the second statement, we use the inclusion/exclusion principle:
    \[k^{(N)}_{{A}}(\pi_0;\pi_1,\dots,\pi_h) =
    \sum_{I \subseteq [h]} (-1)^I k^{\bm{A}} \left( \pi_0 \vee \left(
    \bigvee_{i \in I} \pi_i \right) \right).\]
    Then the second item follows from the first.
\end{proof}

Let us come back to the proof of Theorem \ref{ThMain}.
We fix two lists $\ii$ and $\ss$ of length $r$, as well as a set partition
$\tau$ of $r$.
We want to find a bound for 
\[\kappa \left(\prod_{j \in \tau_1} B^{(N)}_{i_j,s_j} ,\dots,
    \prod_{j \in \tau_\ell} B^{(N)}_{i_j,s_j} \right) 
    = \sum_{\pi \in \PPP([r]) \atop \pi \geq \tau} 
    \prod_{C \in \pi} \esper\left( \prod_{i \in C} B^{(N)}_{i_j,s_j} \right). \]
We split the sum according to the values
of the partitions $\pi_1 = \pi \wedge CC(G_1(\ii,\ss))$
and $\pi_2 = \pi \wedge CC(G_2(\ii,\ss))$.
More precisely,
\[\kappa \left(\prod_{j \in \tau_1} B^{(N)}_{i_j,s_j} ,\dots,                        
    \prod_{j \in \tau_\ell} B^{(N)}_{i_j,s_j} \right) =
    \sum_{\pi_1 \leq CC(G_1(\ii,\ss)) \atop \pi_2 \leq CC(G_2(\ii,\ss))}
    Y^{(N)}_{\pi_1,\pi_2},\]
where 
\[Y^{(N)}_{\pi_1,\pi_2} = \sum_{\pi \geq \tau \atop 
{\pi \wedge CC(G_1(\ii,\ss)) =\pi_1 \atop \pi \wedge CC(G_2(\ii,\ss)) =\pi_2}}
\prod_{C \in \pi} \esper\left( \prod_{i \in C} B^{(N)}_{i_j,s_j} \right).\]
We call the summation index the slice determined by $\pi_1$ and $\pi_2$.

Let us fix some partitions $\pi_1$ and $\pi_2$.
For each block $C$ of $\pi_1$, we consider some sequence of random variables
$(A^{(N)}_C)_{N \geq 1}$ such that:
for each list of distinct blocks $C_1$, \dots, $C_h$
\[\esper( A^{(N)}_{C_1} \cdots A^{(N)}_{C_h})=
\frac{1}{(N+\theta-1)(N+\theta-2)\dots(N+\theta-h)}.
\]
For readers which wonder whether such variables exist,
we refer to Remark \ref{RemUmbral}, which remains valid here.
Consider the family
\begin{equation}\label{EqFamily}
    \left( (N+\theta-1) A^{(N)}_C \right)_{C \in \pi_1}.
\end{equation}
By the same argument as in Section \ref{SubsecProofDistinct},
this family has the quasi-factorization property 
 and,
hence, its cumulants and truncated cumulants are small 
(Lemma \ref{LemMultiplicativeCritetion}).

But, if $\pi$ is in the slice determined by $\pi_1$ and $\pi_2$,
one can check easily (see the description of joint moments in Section
\ref{SubsectJointMoments})
that the corresponding product of moments is given by:
\[\prod_{C \in \pi} \esper\left( \prod_{i \in C} B^{(N)}_{i_j,s_j} \right)=
\alpha_{\pi_1,\pi_2} \prod_{C \in \pi} \esper\left( 
\prod_{C' \in \pi_1 \atop C'\subseteq C} A^{(N)}_{C'} \right),\]
where $\alpha_{\pi_1,\pi_2}$ depends only on $\pi_1$ and $\pi_2$
and is given by:
\begin{itemize}
    \item $0$ if $\pi_2$ contains in the same block two indices $j$ and $h$
        such that $i_j=i_h$ but $s_j \neq s_h$ or $s_j=s_h$ but $i_j \neq i_h$;
    \item $\theta^\gamma$ otherwise, where $\gamma$ is the number of cycles
        of the partial permutation $(\ii,\ss)$, whose indices are all contained
        in the same block of $\pi_2$.
\end{itemize}
As a consequence,
\begin{equation}\label{EqSlice}
    Y^{(N)}_{\pi_1,\pi_2} = \frac{\alpha_{\pi_1,\pi_2}}{(N+\theta-1)^{\#(\pi_1)}}
\sum_{\pi \geq \tau \atop 
{\pi \wedge CC(G_1(\ii,\ss)) =\pi_1 \atop \pi \wedge CC(G_2(\ii,\ss)) =\pi_2}}
\prod_{C \in \pi} \esper\left( 
\prod_{C' \in \pi_1 \atop C'\subseteq C} (N+\theta-1) A_{C'} \right).
\end{equation}
But the condition $\pi \wedge CC(G_1(\ii,\ss)) =\pi_1$ can be rewritten
as follows: $\pi \geq \pi_1$ and $\pi \ngeq \pi'$
for any $\pi_1 \leq \pi' \leq CC(G_1(\ii,\ss))$.
A similar rewriting can be performed for the condition 
$\pi \wedge CC(G_2(\ii,\ss)) =\pi_2$.
Finally, the sum in equation \eqref{EqSlice} above is a truncated cumulant
of the family \eqref{EqFamily} and is bounded from above by
$O(N^{-|CC(G_2(\ii,\ss)) \vee \tau|+1}).$
This implies 
\[Y^{(N)}_{\pi_1,\pi_2} = O(N^{-\#(\pi_1)-|CC(G_2(\ii,\ss)) \vee \tau| +1}),\]
which ends the proof of Theorem \ref{ThMain} because $\pi_1$ has necessarily
at least as many parts as $CC(G_1(\ii,\ss))$.\qed

\begin{remark}
So far, we have considered the lists $\ii$ and $\ss$ as fixed.
Therefore, the constant hidden in the Landau symbol $O$ may depend on these lists.
However, the quantity for which we establish an upper bound depends only
on the partition $\tau$ and on which entries of the lists $\ii$ and $\ss$ coincide.
For a fixed $r$, the number of partitions and of possible equalities is finite.
Therefore, we can choose a constant depending only on $r$,
as it is done in the statement of Theorem \ref{ThMain}.
\end{remark}

\section{Graph-theoretical lemmas}\label{SectGraph}

In this section, we present two quite easy lemmas on the number of connected
components on graph quotients.
These lemmas may already have appeared in the literature,
though the author has not been able to find a reference.
They will be useful in the next sections for applications of
Theorem \ref{ThMain}.

\subsection{Notations}
Let us consider a graph $G$ with vertex set $V$ and edge set $E$.
By definition, if $V'$ is a subset of $V$, the graph $G[V']$ {\em induced}
by $G$ on $V'$ has vertex set $V'$ and edge set $E[V']$,
where $E[V']$ is the subset of $E$ consisting of
edges having both their extremities in $V'$.

Let $f$ be a surjective map from $V$ to another set $W$.
Then the {\em quotient} of $G$ by $f$ is the graph $G/f$ with vertex set $W$
and which has an edge between $w$ and $w'$ if, in $G$,
there is at least one edge between a vertex of $f^{-1}(w)$ and
a vertex of $f^{-1}(w')$.

{\em Example.} Consider the graph $G$ on the top of figure \ref{FigExGraph}.
Its vertex set is the $10$-element set $V=\{1,2,3,4,5,\bar{1},\bar{2},\bar{3},\bar{4},\bar{5}\}$.
Consider the application $f$ from $V$ to the set $ W=\{1,2,3,4,5\}$, consisting
in forgetting the bar (if any).
The contracted graph $G/f$ is drawn on the bottom left picture of
Figure \ref{FigExGraph}.
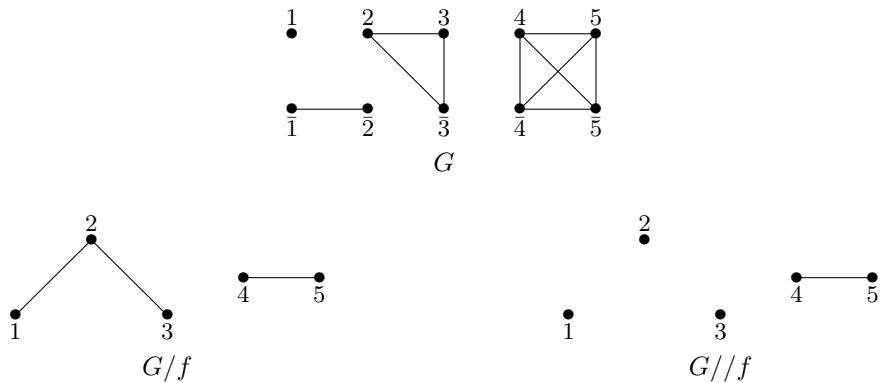
\begin{figure}
\begin{center}
     \begin{tikzpicture}
     	\foreach \i in {1,...,5} {\draw (\i,0) node {$\bullet$} node[anchor=south] {{\small $\i$}};}
     	\foreach \i in {1,...,5} {\draw (\i,-1) node {$\bullet$} node[anchor=north] {{\small $\bar{\i}$}};}
	\draw (3,0) -- (2,0) -- (3,-1) -- cycle;
	\draw (5,0) -- (4,0) -- (4,-1) -- (5,-1) -- (5,0) -- (4,-1);
    \draw (5,-1) -- (4,0);
    \draw (1,-1) -- (2,-1);
	\draw (3,-1.7) node {$G$};
     \end{tikzpicture}
\end{center}
\begin{minipage}{.48\linewidth}
\begin{center}
     \begin{tikzpicture}
	\draw (0,0) node {$\bullet$} node[anchor=north] {\small $1$} 
	-- (1,1) node {$\bullet$} node[anchor=south] {\small $2$}
	-- (2,0) node {$\bullet$} node[anchor=north] {\small $3$};
	\draw (3,.5) node {$\bullet$} node[anchor=north] {\small $4$} 
	-- (4,.5) node {$\bullet$} node[anchor=north] {\small $5$} ;
	\draw (2, -.7) node {$G/f$};
     \end{tikzpicture}
\end{center}
\end{minipage} \hfill
\begin{minipage}{.48\linewidth}
\begin{center}
     \begin{tikzpicture}
	\draw (0,0) node {$\bullet$} node[anchor=north] {\small $1$} ;
	\draw (1,1) node {$\bullet$} node[anchor=south] {\small $2$} ;
	\draw (2,0) node {$\bullet$} node[anchor=north] {\small $3$} ;
	\draw (3,.5) node {$\bullet$} node[anchor=north] {\small $4$} 
	-- (4,.5) node {$\bullet$} node[anchor=north] {\small $5$} ;
	\draw (2, -.7) node {$G//f$};
     \end{tikzpicture}
\end{center}
\end{minipage} 
\caption{An example of a graph, its quotient and strong quotient.}
\label{FigExGraph}
\end{figure}

\subsection{Connected components of quotients}

\begin{lemma}
    Let $G$ be a graph with vertex set $V$ and $f$ a surjective map
    from $V$ to another set $W$.
    Then
    \[ \#(\Conn(G)) \leq \#(\Conn(G/f)) +
        \sum_{w \in W} (\#(\Conn(G[f^{-1}(w)])) - 1).\]
    \label{LemContractionFibers}
\end{lemma}
\begin{proof}
    For each edge $(w,w')$ in $G/f$, we choose arbitrarily an edge $(v,v')$
    in $G$ such that $f(v)=w$ and $f(v')=w'$ (by definition of $G/f$, such an
    edge exists but is not necessarily unique).
    Thereby, to each edge of $G/f$ or of $G[f^{-1}(w)]$ (for any $w$ in $W$)
    corresponds canonically an edge in $G$.

    Take spanning forests $F_{G/f}$ and $(F_w)_{w \in W}$
    of graphs $G/f$ and $G[f^{-1}(w)]$ for $w \in W$.
    With the remark above, to each spanning forest corresponds a set of edges
    in $G$.
    Consider the union $F$ of these sets.
    It is an acyclic set of edges of $G$.
    Indeed, if it contained a cycle, it must
    be contained in one of the fibers $f^{-1}(w)$,
    otherwise it would induce a cycle in $F_{G/f}$.
    But, in this case, all edges of the cycles belong to $F_w$,
    which is impossible, since $F_w$ is a forest.

    Finally, $F$ is an acyclic set of edges in $G$ and
    \begin{multline*}
        \#(\Conn(G)) \leq |V| - |F| = |W| - |F_{G/f}| + 
    \sum_{w \in W} (|f^{-1}(W)| - 1 - |F_w|) \\ \leq \#(\Conn(G/f)) +               
        \sum_{w \in W} (\#(\Conn(G[f^{-1}(w)])) - 1). \qedhere
    \end{multline*}
\end{proof}

{\em Continuing the example.} All fibers $f^{-1}(i)$ (for $i=1,2,3,4,5$)
are of size $2$.
Three of them contains one edge (for $i=3,4,5$) and hence are connected,
while the other two have two connected components.
Finally, the sum in the lemma is equal to $2$, which is equal to the difference
\[ \#(\Conn(G)) - \#(\Conn(G/f)) = 4 - 2 = 2.\]

\subsection{Fibers of size 2}
\label{SubsectFibers2}

In this section, we further assume that
$V = W \sqcup W$ and that $f$ is the canonical application
$W \sqcup W \to W$ consisting in forgetting to which copy of $W$
the element belongs.
Throughout the paper, for simplicity of notation,
we will use overlined letters for elements
of the second copy of $W$.

In this context, in addition to 
the quotient $G/f$, one can consider another graph
with vertex set $W$.
By definition, $G//f$ has an edge between $w$ and $w'$
if, in $G$, there is an edge between $w$ and $w'$
{\em and} an edge between $\bar{w}$ and $\bar{w'}$.
We call this graph the {\em strong quotient} of $G$.

{\em Continuing the example.} The graph $G$ and the function
$f$ in the example above fit in the context described in this section.
The strong quotient $G//f$ is drawn on Figure \ref{FigExGraph}
(bottom right picture).

\begin{lemma}
Let $G$ and $f$ be as above.
Then
\[ \#(\Conn(G)) \leq \#(\Conn(G/f)) + \#(\Conn(G//f)).\]
    \label{LemGraphFibers2}
\end{lemma}
\begin{proof}
    Set $G_1=G//f$, $G_2=G/f$.

    By definition, an edge in $G_1$ between $j$ and $k$
    corresponds to two edges in $G$.
    In contrast, an edge $(i,j)$ in $G_2$ corresponds
    to at least one edge in $G$.

    Consider a spanning forest $F_1$ in $G_1$.
    As the set of edges of $G_1$ is smaller than the one of $G_2$,
    $F_1$ can be completed into a spanning forest $F_2$ of $G_2$.
    We consider the subset $F$ of edges of $G$ obtained as follows:
    for each edge of $F_1$, we take the two corresponding edges in $G$
    and for each edge of $F_2 \backslash F_1$, we take the corresponding edge
    in $G$ (if there is several corresponding edges, choose one arbitrarily).
    \medskip

    We will prove by contradiction that $F$ is acyclic.
    Suppose that $F$ contains a cycle $C$.
    Each edge of $C$ projects on an edge in $F_2$ and thus the 
    projection of $C$ is a list $S=(e_1,\dots,e_h)$
    of {\em consecutive} edges in $F_2$
    ({\em consecutive} means that we can orient the edges so that,
    for each $\ell \in [h]$,
    the end point of $e_\ell$ is the starting point of $e_{\ell+1}$,
    with the convention $e_{h+1}=e_1$).
    This list is not necessarily a cycle because it can contain
    twice the same edges (either in the same direction or in different directions).
    Indeed, $F$ contains some pairs of edges of the form 
    \[ \big( \{w,w'\} , \{\overline{w},\overline{w'}\} \big) \]
    which project on the same edge in $G_2$.
    But as edges from these pairs have no extremities in common, they can
    not appear consecutively in the cycle $C$.
    Therefore, the same edge can not appear twice in a row in the list $S$.
    This implies that the list $S$ contains a cycle $C_2$ as a factor.
    We have reached a contradiction as the edges in $C_2$ are edges of 
    the forest $F_2$.
    Thus $F$ is acyclic.
    \medskip

    The number of edges in $F$ is clearly $2|F_1|+|F_2 \setminus F_1|=|F_1|+|F_2|$.
    Therefore
    \[
        \#(\Conn(G)) \leq 2|W| - |F| = (|W| - |F_1|) + (|W| - |F_2|) 
    = \#(\Conn(G_1)) + \#(\Conn(G_2)). \qedhere
    \]
\end{proof}

\section{Toy example: number of cycles of a given length $p$}   \label{SectCycles}
In this section, we are interested in the number $\Gamma_p^{(N)}$ of cycles of length $p$
in a random Ewens permutation of size $N$.
The asymptotic behavior of $\Gamma_p^{(N)}$ is easy to determine (see Theorem \ref{ThmCycles}),
as its generating series is explicit and quite simple.
We will give another proof which relies on Theorem \ref{ThMain} and
does not use an explicit expression for the generating series of $\Gamma_p^{(N)}$.

The main steps of the proof are the same in the other examples,
so let us emphasize them here.\medskip

{\em Step 1: expand the cumulants of the considered statistic.}

In this step, one has to express the statistic we are interested in
using the variables $B^{(N)}_{i,s}$: here,
\[\Gamma_p^{(N)} = \sum_{1 \leq i_1 < i_2,i_3,\dots,i_p \leq N} \Bc_{(i_1,\dots,i_p)},\]
where $\Bc_{(i_1,\dots,i_p)}=B^{(N)}_{i_1,i_2} \dots B^{(N)}_{i_{p-1}, i_p} B^{(N)}_{i_p,i_1}$
is the indicator function of the event ``$(i_1,\dots,i_p)$ is a cycle of $\sigma$''.
Therefore, one has
\begin{equation}\label{EqDevCumCp}
    \kappa_\ell(\Gamma_p^{(N)}) = \sum_{ {i^1_1 < i^1_2,i^1_3,\dots,i^1_p \atop \vdots}
\atop i^\ell_1 < i^\ell_2,i^\ell_3,\dots,i^\ell_p}
\kappa\big( B^{(N)}_{i^1_1,i^1_2} \dots B^{(N)}_{i^1_p,i^1_1}, \cdots, 
 B^{(N)}_{i^\ell_1,i^\ell_2} \dots B^{(N)}_{i^\ell_p,i^\ell_1} \big).
 \end{equation}

{\em Step 2: Give an upper bound for the elementary cumulants.}

Now, we would like to apply our main lemma to every summand of equation 
\eqref{EqDevCumCp}.
To this purpose, one has to understand what is the exponent of $N$ in the upper bound
given by Theorem \ref{ThMain}.

For a matrix 
\[(i_j^r)_{1 \leq j \leq p \atop 1 \leq r \leq \ell},\]
we denote:
\begin{itemize}
 \item $M(\ii)=|\{(i^r_j,i^r_{j+1}); 1 \leq j \leq p, 1 \leq r \leq \ell\}|$
     the number of different entries in the matrix of pairs $(i^r_j,i^r_{j+1})$
      (by convention, $i^r_{p+1}=i^r_1$);
 \item $Q(\ii)$ the number of connected components of the graph 
     $G(\ii)$ on $[\ell]$
     where $r_1$ is linked with $r_2$ if
     \[ \{i_j^{r_1}; 1\leq j \leq p\} \cap \{i_j^{r_2}; 1\leq j \leq p\}
     \neq \emptyset;\]
 \item $t(\ii)$ the number of distinct entries.
\end{itemize}
Clearly, $M(\ii)$ is always at least equal to $t(\ii)$.
In the case where $\tau$ has $\ell$ blocks of size $p$ and where the list $\ss$ is obtained
by a cyclic rotation of the list $\ii$ in each block, Theorem \ref{ThMain} writes as:
\begin{multline}\label{EqBoundCumCp}
    \big| \kappa\big( B^{(N)}_{i^1_1,i^1_2} \dots B^{(N)}_{i^1_p,i^1_1}, \cdots,  
B^{(N)}_{i^\ell_1,i^\ell_2} \dots B^{(N)}_{i^\ell_p,i^\ell_1} \big) \big| \leq C_{p\ell}
N^{-M(\ii)-Q(\ii)+1} \\
\leq C_{p\ell} N^{-M(\ii)}
\leq C_{p\ell} N^{-t(\ii)}.
\end{multline}

{\em Step 3: give an upper bound for the number of lists.}

As the number of summands in Equation~\eqref{EqDevCumCp} depends on $N$, 
we can not use directly inequality~\eqref{EqBoundCumCp}.
We need a bound on the number of matrices $\ii$ with a given value of $t(\ii)$.

This bound comes from the following simple lemma:

\begin{lemma}
    For each $L \geq 1$, there exists a constant $C'_L$ with
     the following property.
    For any $N \geq 1$ and $t \in [L]$, the number of lists
    $\ii$ of length $L$ with entries in $[N]$ such that
    \[|\{i_1,\dots,i_L\}|=t\]
    is bounded from above by
    $C'_L N^t$.
    \label{LemNumberLists}
\end{lemma}
\begin{proof}
    If we specify which indices correspond to entries with the same values
    (that is a set partition in $t$ blocks of the set of indices),
    the number of corresponding
    lists is $\binom{N}{t}$ and hence is bounded from above by $N^t$.
    This implies the lemma, with $C'_L$ being equal to 
    the number of set partitions of $[L]$.
\end{proof}

{\em Step 4: conclude.}

By inequality~\eqref{EqBoundCumCp} and Lemma \ref{LemNumberLists},
for each $t \in [p \cdot \ell]$,
the contribution of lists $(i_j^r)$ taking exactly $t$ different values is bounded from above
by $C'_{p \ell} C_{p \ell}$ and hence
\[ \text{for all }\ell \geq 1, \kappa_\ell(\Gamma_p^{(N)})=O(1).\]

To compute the component of order $1$, let us make the following remark:
by the argument above, the total contribution of lists $(i_j^r)$
with $M(\ii) > t(\ii)$ or $Q(\ii)>1$ is $O(N^{-1})$.

But $M(\ii)=t(\ii)$ implies that, as soon as 
\[\{i_j^{r_1}; 1\leq j \leq p\} \cap \{i_j^{r_2}; 1\leq j \leq p\}
\neq \emptyset,\]
the cyclic words $(i^{r_1}_1,\dots,i^{r_1}_p)$ and $(i^{r_2}_1,\dots,i^{r_2}_p)$
are equal.
As $i^r_1$ is always the minimum of the $i^r_j$, the two words are 
in fact always equal in this case.
In particular $G(\ii)$ is a disjoint union of cliques.
If we further assume $Q(\ii)=1$, {\it i.e.} $G(\ii)$ is connected,
then $G(\ii)$ is the complete graph and we get that $i_j^r$ does not depend on $r$.

Finally
\begin{equation}\label{EqCumCp}
    \kappa_\ell(\Gamma_p^{(N)}) = \sum_{i_1 < i_2,i_3,\dots,i_p}
 \kappa_\ell \big(B^{(N)}_{i_1,i_2} \dots B^{(N)}_{i_p,i_1}\big) + O(N^{-1}).
 \end{equation}
 But each $B^{(N)}_{i_1,i_2} \dots B^{(N)}_{i_p,i_1}$ is a Bernoulli variable
 with parameter $\theta/(N+\theta-1)_p$.
 Therefore their moments are all equal to $\theta/(N+\theta-1)_p$ and by 
 formula~\eqref{EqMoment2Cumulant}, their cumulants are $\theta/(N+\theta-1)_p + O(N^{-2p})$.
 Finally, as there are $(N)_p/p$ terms in equation~\eqref{EqCumCp},
 \[\kappa_\ell(\Gamma_p^{(N)}) = \frac{\theta}{p} + O(N^{-1}),\]
 which implies that $\Gamma_p^{(N)}$ converges in distribution towards
 a Poisson law with parameter $\frac{\theta}{p}$.
 
 Moreover, a simple adaptation of the proof of Equation~\eqref{EqCumCp} implies
 that
 \[\kappa(\Gamma_{p_1}^{(N)},\dots,\Gamma_{p_\ell}^{(N)})= O(N^{-1})\]
 as soon as two of the $p_r$'s are different.
 Indeed, no matrices $(i_j^r)_{1 \leq r \leq \ell \atop 1 \leq j \leq p_r}$
 with rows of different sizes fulfill simultaneously $M(\ii)=t(\ii)$ and $Q(\ii)=1$.
Finally, for any $p\geq 1$, the vector $(\Gamma_1^{(N)},\dots,\Gamma_p^{(N)})$
tends in distribution towards a vector $(P_1,\dots,P_p)$
where the $P_i$ are independent Poisson-distributed random variables
with respective parameters $\theta/i$.\qed

\section{Number of exceedances}   \label{SectExc}

In this section, we look at our second motivating problem,
the number of excee\-dances in random Ewens permutations.
The first two subsections make a link between a physical statistics model
and this problem, justifying our work.
The last two subsections are devoted to the proof of Theorem \ref{ThmExc} and
related results.

\subsection{Symmetric simple exclusion process}
\label{SubsectSSEP}
The symmetric simple exclusion process ({\em SSEP} for short) is a model of statistical
physics:
we consider particles on a discrete line with $N$ sites.
No two particles can be in the same site at the same moment.
The system evolves as follows:
\begin{itemize}
    \item if its neighboring site is empty, a particle can jump to its left or
        its right with probability $\frac{1}{N+1}$;
    \item if the left-most site is empty (resp. occupied), a particle can enter (resp. leave)
        from the left with probability $\frac{\alpha}{N+1}$ (resp. $\frac{\gamma}{N+1}$);
    \item if the right-most site is empty (resp. occupied), a particle can enter (resp. leave)
        from the right with probability $\frac{\delta}{N+1}$ (resp. $\frac{\beta}{N+1}$);
    \item with the remaining probability
        (we suppose $\alpha,\beta,\gamma,\delta <1$ so that,
        in a given state, the sum of the probabilities of the events which may occur
        is smaller than $1$), nothing happens.
\end{itemize}
Mathematically, this defines an irreducible aperiodic Markov chain on the finite set
$\{0;1\}^N$ (a state of the {\em SSEP} can be encoded as a word in $0$ and $1$ of length $N$,
where the entries with value $1$ correspond to the positions of the occupied sites).

This model is quite popular among physicists because, despite its simplicity,
it exhibits interesting phenomenons like the existence of different phases.
For a comprehensive introduction on the subject and a survey of results,
see \cite{DerridaSurveyPASEP}.

A good way to describe a state $\tau$ of the {\em SSEP} is the function
$F^{(N)}_\tau$ defined as follows: when $Nx$ is an integer,
\[F^{(N)}_\tau(x)=\frac{1}{N} \cdot \sum_{i=1}^{Nx} \tau_i \]
and, for each $i \in [N]$, the function $F^{(N)}_\tau$ is affine between
$(i-1)/N$ and $i/N$.
One should see $F^{(N)}_\tau$ as the integral of the density of particles in the system.

We are interested in the steady state (or stationary distribution) of the {\em SSEP},
that is the unique probability measure $\mu_N$ on $\{0;1\}^N$,
which is invariant by the dynamics.
More precisely, we want to study asymptotically the properties of the random function
$F^{(N)}_\tau$,
where $\tau$ is distributed with $\mu_N$ and $N$ tends to infinity.

\subsection{Link with permutation tableaux and Ewens measure}
\label{SubsectPermTableaux}
From now on, we restrict to the case $\alpha=1$, $\gamma,\delta=0$.
In this case, thanks to a result of S. Corteel and L. Williams
\cite{CorteelWilliamsPASEPPermTableaux},
the measure $\mu_{N}$ is related to some combinatorial objects,
called permutation tableaux.

The latter are fillings of Young diagrams (which can have empty rows,
but no empty columns) with $0$ and $1$ respecting some rules, the details 
of which will not be important here.
The Young diagram is called the shape of the permutation tableau.
The size of a permutation tableau is its number of rows, plus its number
of columns (and not the number of boxes!).

In addition with their link with statistical physics, permutation tableaux 
also appear in algebraic geometry: they index the cells of some canonical
decomposition of the totally positive part of the Grassmannian
\cite{PostnikovGrassmannian,WilliamsGrassmann}.
They have also been widely studied from a purely combinatorial point of view
\cite{SteingrimssonWilliamsPermutationTableaux,CorteelNadeauBijectionsPermTab,TLT}.

To a permutation tableau $T$ of size $N+1$, one can associate a word $w^T$ in $\{0;1\}^N$
as follows:
we label the steps of the border of the tableau starting from the North-East corner
to the South-West corner.
The first step is always a South step.
For the other steps, we set $w^T_i=1$ if and only if the $i+1$-th step is a south step.
Clearly, the word $w^T$ depends only on the shape of the tableau $T$.
This procedure is illustrated on figure \ref{FigFromTableauxToState}.

\begin{figure}[t]
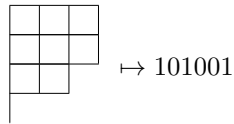

    \begin{center}
        $\begin{array}{c} \yng(3,3,2,0)\end{array} \mapsto 101001 $
    \end{center}
    \caption{From the shape of a permutation tableau to a word in $\{0;1\}^{N-1}$.}
    \label{FigFromTableauxToState}
\end{figure}

With this definition, the border of a tableau $T$ of size $N+1$
is the parametric broken line
\[ \big\{\big(n_1(w^T) - N F^{(N)}_{w^T}(x), - N (x - F^{(N)}_{w^T}(x)) -1 \big) : x \in [0;1] \big\},\]
where $n_1(w^T)$ is the number of $1$ in $w^T$ and $F^{(N)}_{w^T}$ the function
associated to the word $w^T$ as defined in the previous section.
Hence, $F^{(N)}_{w^T}$ is a good way to encode the shape of the permutation tableau $T$.

S. Corteel and L. Williams also introduced a statistics on permutation
tableaux called {\em number of unrestricted rows} and denoted $u(T)$.
If $\beta$ is a positive real parameter, this statistics induces a measure
$\mu^T_N(\beta)$ on permutation tableaux of size $N$,
for which the probability to pick a tableau $T$ is proportional       
to $\beta^{-u(T)}$.                             
This measure is related to the {\em SSEP} by the following result 
(which is in fact a particular case of
\cite[Theorem 3.1]{CorteelWilliamsPASEPPermTableaux}
but we do not know how to deal with the extra parameters there).
\begin{theorem}{\cite{CorteelWilliamsPASEPPermTableaux}}
    The steady state of the {\em SSEP} $\mu_N$ is
    the push-forward by the application $T \mapsto w_T$
    of the probability measure $\mu^T_{N+1}(\beta)$.
    \label{ThmPasepPermTab}
\end{theorem}

It turns out that this measure can also be described using random permutations.
Indeed, S. Corteel and P. Nadeau \cite[Theorem 1 and Section 3]{CorteelNadeauBijectionsPermTab}
have exhibited a simple bijection $\Phi$ between permutations
of $N+1$ and permutation tableaux of size $N+1$, which satisfies:
\begin{itemize}
    \item If a permutation $\sigma$ is mapped to a tableau $T=\Phi(\sigma)$, then:
        \[w^T = (\delta_2(\sigma),\delta_3(\sigma),\dots, \delta_{N+1}(\sigma)),\]
        where $\delta_i=1$ if $i$ is an ascent, that is if $\sigma(i) < \sigma(i+1)$
        (by convention $\delta_{\sigma(N+1)}(\sigma)=1$).
    \item The number of unrestricted rows of a tableau $T=\Phi(\sigma)$ is the number of
        right-to-left minima of $\sigma$: recall that $i$ is
        a right-to-left minimum of $\sigma$ if 
        $\sigma_\ell > i$ for any $\ell>\sigma^{-1}(i)$.
\end{itemize}

We are interested in the number of cycles of permutations rather than
their number of right-to-left minima.
The following bijection, which is a variant of the first fundamental
transformation on permutation \cite[§ 10.2]{LothaireCombinatorics},
sends one of this statistics to the other.
Take a permutation $\sigma$, written in its cycle notation so that:
\begin{itemize}
    \item its cycles end with their minima;
    \item the minima of the cycles are in increasing order.
\end{itemize}
For example, $\sigma=(3\ 5\ 1) (7\ 4\ 2) (6).$
Now, erase the parenthesis: we obtain the word notation of a permutation $\Psi(\sigma)$.

The application $\Psi$ is a bijection from $S_N$ to $S_N$.
Besides, the minima of the cycles of $\sigma$ are the right-to-left minima of $\Psi(\sigma)$,
while the ascents in $\Psi(\sigma)$ are the exceedances in $\sigma$
(a similar statement is given in \cite[Theorem 10.2.3]{LothaireCombinatorics}).

From now on, we assume $\beta\cdot \theta=1$.
The properties above imply that $\mu^T_N(\beta)$ is the push-forward of the Ewens
measure with parameter $\theta$ by the application
$\Phi \circ \Psi$.
Combining this with Theorem \ref{ThmPasepPermTab},
the steady state of the {\em SSEP} $\mu_N$ is the push-forward of Ewens
measure by the application $\sigma \mapsto w^{\Phi(\Psi(\sigma))}$.
But this application admits an easy direct description
\[ \begin{array}{rcl}
    S_{N+1} & \to & \{0;1\}^N \\
    \sigma & \mapsto & (\delta_{\sigma(2) \geq 2},\delta_{\sigma(3) \geq 3},\dots,
    \delta_{\sigma(N+1) \geq N+1}).
\end{array} \]

Recall that, as explained above, we are interested in the random function $F^{(N)}_\tau$,
where $\tau$ is distributed according to the measure $\mu_{N-1}$.
The results above imply that this random function has the same distribution
as $F^{(N+1)}_\sigma$,
where $\sigma$ is a random Ewens permutation of size $N$
 and $F^{(N+1)}_\sigma$ is the function defined in Section \ref{SubsectMotivations}.

 This was our original motivation to study $F^{(N+1)}_\sigma$.

\subsection{Bounds for cumulants}
Let us fix some real numbers $x_1,\dots,x_\ell$ in $[0;1]$.
In this section, we will give some bounds on the joint cumulants of the
random variables $(F^{(N)}_\sigma(x_1),\dots,F^{(N)}_\sigma(x_\ell))$.

Let us begin by the following bound (step 2 of the proof, according to the
division done in Section \ref{SectCycles}).
\begin{proposition}\label{PropBoundTau0}
 For any $\ell \geq 1$, 
any $N \geq 1$ and any lists $i_1,\dots,i_\ell$ and $s_1,\dots,s_\ell$
of integers in $[N]$,
\[ \kappa(B^{(N)}_{i_1,s_1},\dots,B^{(N)}_{i_\ell,s_\ell}) \leq C_\ell
N^{-|\{i_1,\dots,i_\ell,s_1,\dots,s_\ell\}|+1}, \]
where $C_\ell$ is the constant defined by Theorem~\ref{ThMain}.
\end{proposition}
\begin{proof}
    Using Theorem \ref{ThMain} for $\tau=\big\{\{1\},\dots,\{\ell\}\big\}$,
    we only have to prove that
    \[-\#\big(\Conn(G_1(\ii,\ss))\big) - \#\big(\Conn(G_2(\ii,\ss)) \big) \geq 
   - |\{i_1,\dots,i_\ell,s_1,\dots,s_\ell\}|.\]
    The last quantity $|\{i_1,\dots,i_\ell,s_1,\dots,s_\ell\}|$ can be seen as
    the number of connected components of the graphs $G(\ii,\ss)$ defined as follows:
    \begin{itemize}
        \item its vertex set is $[\ell] \sqcup [\ell] = \{1,\bar{1},\dots,\ell,\bar{\ell}\}$;
        \item there is an edge between $j$ and $k$ (resp. $j$ and $\bar{k}$,
            $\bar{j}$ and $\bar{k}$) if and only if $i_j=i_k$
            (resp. $i_j=s_k$, $s_j=s_k$).
    \end{itemize}
    The inequality above is simply
    Lemma~\ref{LemGraphFibers2} applied to the graph $G(\ii,\ss)$
    ($G_1(\ii,\ss)$ and $G_2(\ii,\ss)$ are respectively its strong and usual quotients).
\end{proof}

We can now prove the following bound:
\begin{proposition}\label{PropCumulantF}
    There exists a constant $C''_\ell$ such that, for any integer $N\geq 1$ and 
    real numbers $x_1$, \dots, $x_\ell$, one has
    \[ |\kappa(F^{(N)}_\sigma(x_1),\dots,F^{(N)}_\sigma(x_\ell))| \leq C''_\ell N^{-\ell+1}.\]
\end{proposition}
\begin{proof}
    To simplify the notations, we suppose that $N x_1,\dots,N x_\ell$ are integers,
    so that
    \[ (N-1) \cdot F^{(N)}_\sigma(x_i) = \sum_{i=2}^{N x_i} \Bex_i(\sigma).\]
    But the Bernoulli variable $\Bex_i$ can be written as
    $\Bex_i= \sum_{s \geq i} B^{(N)}_{i,s}$.
    Finally, by multilinearity, one has (step 1):
    \begin{equation}\label{EqDevCumExc}(N-1)^\ell \kappa(F^{(N)}_\sigma(x_1),\dots,F^{(N)}_\sigma(x_\ell))
    = \sum_{ {2\leq i_1 \leq Nx_1 \atop \vdots} \atop 2\leq i_\ell \leq Nx_\ell}
    \sum_{ {s_1 \geq i_1 \atop \vdots} \atop s_\ell \geq i_\ell}
    \kappa(B^{(N)}_{i_1,s_1},\dots,B^{(N)}_{i_\ell,s_\ell}).
    \end{equation}
    We apply Lemma \ref{LemNumberLists} to the list $i_1,\dots,i_\ell,s_1,\dots,s_\ell$
    and get that the number of pairs of lists $(\ii,\ss)$ such that
    $|\{i_1,\dots,i_\ell,s_1,\dots,s_\ell\}|$ is equal to a given number $t$
    is bounded from above by $C'_{2\ell} N^t$ (step 3).

    Combining this with Proposition \ref{PropBoundTau0}, we get that the total contribution
    of pairs of lists $(\ii,\ss)$ with $|\{i_1,\dots,i_\ell,s_1,\dots,s_\ell\}|=t$
    to the right-hand side of~\eqref{EqDevCumExc} is smaller than
    $C'_{2\ell} C_\ell N$ , which ends the proof of Proposition \ref{PropCumulantF} (step 4).
\end{proof}

{\em Illustration of the proof.}
Set $\ell=5$ and consider the lists $\ii=(5,2,2,7,7)$ and $\ss=(8,8,2,7,7)$.
The graph $G(\ii,\ss)$ associated to this pair of lists
is the graph $G$ drawn of Figure \ref{FigExGraph}.
It follows immediately that $G_1(\ii,\ss)= G//f$ has 4
connected components while $G_2(\ii,\ss)=G/f$ has 2.
Therefore, by Theorem \ref{ThMain},
\[
    \kappa(B^{(N)}_{5,8},B^{(N)}_{2,8},B^{(N)}_{2,2},B^{(N)}_{7,7},B^{(N)}_{7,7})
\leq C_5 N^{-5}.
\]
The same bound is valid for all sequences $\ii$ and $\ss$ such that $G(\ii,\ss)=G$.
There are fewer than $N^4$ such sequences: to construct such a sequence,
one has to choose distinct values for the four connected components of $G$, so that
they fulfill some inequalities.
Finally, their total contribution to~\eqref{EqDevCumExc} is smaller than $C_5 N^{-1}$.
\bigskip

{\em Comparison with a result of B. Derrida, J.L. Lebowitz and E.R. Speer.}
In \cite[Appendix A]{DerridaLongRangeCorrelation}, a {\em long range 
correlation phenomenon} for the {\em SSEP} is proved.
Rewritten in terms of Ewens random permutations {\it via} the material
of the previous section and with mathematical terminology,
it asserts that, for $i_1<\dots<i_\ell$,
\[\kappa(\Bex_{i_1},\dots,\Bex_{i_\ell}) = O(N^{-\ell+1}). \]
In fact, their result is more general because it corresponds to the {\em SSEP}
with all parameters.
This bound on cumulants can be obtained easily using our Proposition \ref{PropBoundTau0}
and Lemma \ref{LemNumberLists}.
A slight generalization of it (taking into account the case where some $i$'s can be
equal) implies directly Proposition \ref{PropCumulantF}.
Therefore, our method does not give some new results on the {\em SSEP}.
Nevertheless, it was natural to try to understand the long range correlation phenomenon
directly in terms of random permutations and that is what our approach does.

\subsection{Convergence results}\label{SubsectConv}
In this section, we explain how one can deduce from the bound on cumulants,
some results on the convergence of the random function $F^{(N)}_\sigma$,
in particular Theorem \ref{ThmExc}.

In addition to the bounds above, we need equivalents for the
first and second joint cumulants of the $F^{(N)}_\sigma(x)$.
An easy computation gives:
\begin{align*}
    \esper(\Bex_i)&=\frac{N-i+\theta}{N+\theta-1} ;\\
    \Var(\Bex_i)&= \frac{(i-1)(N-i+\theta)}{(N+\theta-1)^2} ;\\
    \Cov(\Bex_i,\Bex_j)&=-\frac{(n-j+\theta)(i-1)}{(N+\theta-1)^2(N+\theta-2)} \text{ for }i<j,
\end{align*}
from which we get the limits:
\begin{align}
    \lim_{N \to \infty} \esper(F^{(N)}_\sigma(x)) &= \int_0^x (1-t) dt + o(1) = \frac{1 - (1-x)^2}{2} 
    \label{EqEsperF};\\
    \lim_{N \to \infty} N \Cov(F^{(N)}_\sigma(x),F^{(N)}_\sigma(y))
    &= \int_0^{\min(x,y)} t (1-t) dt  \label{EqCovF}\\
    \nonumber    & \hspace{-1cm}
    - \int_{0 \leq t \leq x \atop 0 \leq u \leq y} \min(t,u) (1-\max(t,u)) dt du.
\end{align}
We call $K(x,y)$ the right-hand side of the second equation.
We begin with a proof of Theorem \ref{ThmExc},
which describes the asymptotic behavior of $F^{(N)}_\sigma(x)$,
for fixed value(s) of $x$.
\begin{proof}
    Consider the first statement.
    The convergence in probability of $F^{(N)}_\sigma(x)$ towards $1/2 \cdot (1-(1-x)^2)$
    follows immediately from equations
    \eqref{EqEsperF} and \eqref{EqCovF}.
    For the almost-sure convergence, we have to study the fourth centered moment.
    
    From moment-cumulant formula \eqref{EqMoment2Cumulant}
    and using the fact that all cumulants but the first are invariant by a
    shift of the variable,
    \[\esper\left( (F^{(N)}_\sigma(x) - \esper(F^{(N)}_\sigma(x)))^4 \right)=
    \kappa_4 ( F^{(N)}_\sigma(x) ) + 3 (\kappa_2 ( F^{(N)}_\sigma(x) ) )^2.\]
    By proposition \ref{PropCumulantF}, this quantity is bounded from above by $O(N^{-2})$
    and, in particular,
    \[ \sum_{N \geq 1} 
    \esper\left( (F^{(N)}_\sigma(x) - \esper(F^{(N)}_\sigma(x)))^4 \right) < \infty.\]
    The end of the proof is classical.
    First, we inverse the summation and expectation symbols (all quantities are nonnegative).
    As its expectation is finite, the random variable
    \[ \sum_{N \geq 1} (F^{(N)}_\sigma(x) - \esper(F^{(N)}_\sigma(x)))^4\]
   is almost surely finite and hence its general term
   $\big( (F^{(N)}_\sigma(x) - \esper(F^{(N)}_\sigma(x))^4\big)_{N \geq 1}$
   tends almost surely to $0$.
    
    Let us consider the second statement.
    Proposition \ref{PropCumulantF} implies that, for any list $j_1$,\dots,$j_\ell$
    of integers in $[r]$, one has
    \[\kappa(Z^{(N)}_\sigma(x_{j_1}),\dots,Z^{(N)}_\sigma(x_{j_\ell})) = O(N^{-r/2+1}).\]
    In particular, for $r>2$ the left-hand side tends to $0$.
    As the variables $Z^{(N)}_\sigma(x_i)$ are centered, this implies that 
    $(Z^{(N)}_\sigma(x_1),\dots,Z^{(N)}_\sigma(x_r))$ tends towards
    a centered Gaussian vector.
    The covariance matrix is the limit of the covariance of the $Z^{(N)}_\sigma(x_i)$,
    that is $(K(x_i,x_j))$.
\end{proof}

The previous theorem deals with pointwise convergence.
It is also possible to get some results for the  
random functions $(F^{(N)}_\sigma)_{N \geq 1}$.
In the following statement, we consider
convergence in the functional space $(C([0;1]),||\cdot||_\infty)$,
that is uniform convergence of continuous functions.

\begin{theorem}
    Almost surely, the function $F^{(N)}_\sigma$ converges towards
    the function 
    \[x \mapsto 1/2 \cdot (1 - (1-x)^2).\]

    Moreover, the sequence of random functions
    $(x \mapsto Z^{(N)}_\sigma(x))_{N \geq 1}$ converges in distribution towards
    the centered Gaussian process $x \mapsto G(x)$
    with covariance function $\Cov(G(x),G(y))=K(x,y)$.
    \label{ThmCvF}
\end{theorem}
\begin{proof}
    As, for any $N \geq 1$ and any $\sigma \in S_N$, the function
    $x \mapsto F^{(N)}_\sigma(x)$ is non-decreasing,
    the first statement follows easily from the convergence at any fixed $x$.
    The argument can be found for example in a paper
    of J.F. Marckert \cite[first page]{JFPontBrownien},
    but it is so short and simple that we copy it here.
    By monotonicity of $F^{(N)}_\sigma$ and $F$, for any list $(x_i)_{0 \le i \le k}$
    with $0=x_0 < x_1 < \dots < x_k =1$, one has
    \begin{multline*}
        \sup_{x \in [0;1]} |F^{(N)}_\sigma(x) - F(x) | \\
    \leq \max_{0 \leq j < k} \max \big(|F^{(N)}_\sigma(x_{j+1})-F(x_j)|,
    |F^{(N)}_\sigma(x_j)-F(x_{j+1})| \big) \\
    \stackrel{\text{\tiny a.s.}}{\longrightarrow}
    \max_{0 \leq j < k} |F(x_j)-F(x_{j+1})|,
\end{multline*}
    which may be chosen as small as wanted.

    Consider the second statement.
    If the sequence of random function $x \mapsto Z^{(N)}_\sigma(x)$ has a limit,
    its finite-dimensional laws are necessarily the limits of the ones of $Z^{(N)}_\sigma$,
    that is, by Theorem \ref{ThmExc}, Gaussian vectors with covariance
    matrices given by $(K(x_i,x_j))_{1 \leq i,j \leq r}$.
    As a probability measure on $\C([0;1])$ is entirely determined by its finite dimensional laws
    \cite[Example 1.2]{BillingsleyConv},
    one just has to prove that the sequence $x \mapsto Z^{(N)}_\sigma(x)$
    has indeed a limit.
    To do this, it is enough to prove that it is {\em tight}
    \cite[Section 5, Theorems 5.1 and 7.1]{BillingsleyConv},
    that is, for each $\epsilon >0$ there exists some constant $M$ such that:
    \[\text{for all }N>0, \text{ one has }
    \Prob\big( || Z^{(N)}_\sigma ||_\infty >M \big) \le \epsilon.\]

    Once again, this follows from a careful analysis of the fourth moment.
    

    Let $N \geq 1$ and $s \neq s'$ in $[0;1]$ such that $Ns$ and $Ns'$ are integers.
    Using equation \eqref{EqMoment2Cumulant} and
    the fact that $Z^{(N)}_\sigma(s)$ and $Z^{(N)}_\sigma(s')$ are centered,
    one has:
    \begin{multline*}
        \esper\left( ( Z^{(N)}_\sigma(s) - Z^{(N)}_\sigma(s'))^{4} \right) 
        \\ =\kappa_4(Z^{(N)}_\sigma(s) - Z^{(N)}_\sigma(s'))+ 3 \kappa_2(Z^{(N)}_\sigma(s) - Z^{(N)}_\sigma(s'))^2
    \\= N^2 \big(\kappa_4(F^{(N)}_\sigma(s)-F^{(N)}_\sigma(s')) 
        + 3 \kappa_2(F^{(N)}_\sigma(s) - F^{(N)}_\sigma(s'))^2 \big).
    \end{multline*}
    A simple adaptation of the proof of Proposition \ref{PropCumulantF} 
    shows that
    \[\kappa_\ell(F^{(N)}_\sigma(s)-F^{(N)}_\sigma(s')) \leq C_\ell N^{-\ell+1} |s-s'|.\]
    Indeed, in Lemma \ref{LemNumberLists}, if we ask that at least one entry
    of the list $\ii$ is between $Ns$ and $Ns'$ then the number of lists
    is bounded from above by $C'_L N^t |s-s'|$.
    Finally, 
    \begin{multline*}
        \esper\left( ( Z^{(N)}_\sigma(s) - Z^{(N)}_\sigma(s'))^{4} \right) \leq (N^2 (C_4 N^{-3} |s-s'| + 3 C_2^2 N^{-2} |s-s'|^2)) \\
    \leq (C_4 + 3 C_2^2) |s-s'|^2.
    \end{multline*}
    The last inequality has been deduced from $|s-s'| \geq N^{-1}$.
    
    We can now apply Theorem 10.2 of Billingsley's book \cite{BillingsleyConv}
    with $S_i=Z^{(N)}_\sigma(i/N)$ (for $0 \leq i \leq N$), $\alpha=\beta=1$
    and $u_\ell= (C_4 + 3 C_2^2)^{1/2}/N$
    (see equation (10.11) of the same book).
    We get that there exists some constant $K$ such that
    \[ \Prob \big( \max_{0 \leq i \leq N} | S_i |  \geq M \big) \leq K M^{-4},\]
    which proves that the sequence $Z^{(N)}_\sigma$ is tight.
\end{proof}


\section{Generalized patterns}\label{SectPatterns}

This Section is devoted to the applications of our method to 
adjacencies (Subsection \ref{SubsectAdj}) and dashed patterns
(Subsection \ref{SubsectDashedPatterns}).
These two statistics belong in fact to the same general framework and
we discuss in Subsection \ref{SubsectLocalStat} the possibility
of unifying our results.

The proofs in this section are a little bit more technical than
the ones before and in particular we need a new lemma for step 3,
given in Subsection \ref{SubsectLemNbSuites}.

\subsection{Preliminaries}\label{SubsectLemNbSuites}
Let $L \geq 1$ be an integer.
For each pair $\{j,k\} \subset [L]$, we choose a {\em finite} set of
integers $D_{\{j,k\}}$.

Consider a list $i_1,\dots,i_L$ of integers.
For each pair $e=\{j,k\} \subset [L]$ (with $j<k$), we denote $\delta_e(\ii)$
the difference $i_k-i_j$.
Then we associate to $\ii$
a graph of vertex set $[L]$ and edge set $\{e : \delta_e(\ii) \in D_e\}$.

The following lemma is a slight generalization of Lemma \ref{LemNumberLists}

\begin{lemma}
    For each $L$ and family of sets $(D_{\{j,k\}})_{1 \leq j < k \leq L}$,
    there exists a constant $C''_{L,\DD}$ with the following property.
    For any $N\geq 1$ and $t \leq L$, the number of sequences $i_1,\dots,i_L$ with entries
    in $[N]$, whose corresponding graph has exactly $t$ connected components
    is bounded from above by $C''_{L,\DD} N^t$.
    \label{LemNumberListsGen}
\end{lemma}
\begin{proof}
    If we fix a graph $G$ with vertex set $L$ and $t$ connected components and
    if we fix also, for each edge $e$ of the graph, the actual value of $\delta_e(\ii)$,
    then the corresponding number of lists $\ii$ is smaller than $N^t$.
    Indeed, the sequence will be determined by the choice of one value per connected component
    of $G$ (with some constraints, so that no extra edges appear).
    But the number of graphs and of values on edges are finite (the sets $D_{j,k}$ are
    finite) and depend only on $L$ and on the family $\DD$.
\end{proof}

\subsection{Adjacencies}\label{SubsectAdj}
In this section, we prove the following extension of Theorem
\ref{ThmAdj}.

\begin{theorem}
    Let $\sigma_N$ be a sequence of random Ewens permutations, such that $\sigma_N$
    has size $N$.
    Then the number $A^{(N)}$ of adjacencies in $\sigma_N$
    converges in distribution towards a Poisson variable with parameter $2$.
    \label{ThmAdjEwens}
\end{theorem}
\begin{proof}
    As before, we write $A^{(N)}$ in terms of the $B^{(N)}_{i,s}$ (we use the convention
    $B^{(N)}_{i,s}=0$ if $i$ or $s$ is not in $[N]$):
    \[A^{(N)} = \sum_{1 \leq i,s \leq N \atop \epsilon = \pm 1}
    B^{(N)}_{i,s} B^{(N)}_{i+1,s+\epsilon}.\]
    Hence, for $\ell \geq 1$, its $\ell$-th cumulant writes as (step 1):
    \begin{equation}\label{EqDevCumAdj}
    \kappa_\ell(A^{(N)}) = \sum_{1 \leq i_1,s_1,\dots,i_\ell,s_\ell \leq N \atop 
    \epsilon_1,\dots,\epsilon_\ell= \pm 1}
    \kappa \bigg( B^{(N)}_{i_1,s_1} B^{(N)}_{i_1+1,s_1+\epsilon_1}, \cdots,
    B^{(N)}_{i_\ell,s_\ell} B^{(N)}_{i_\ell+1,s_\ell+\epsilon_\ell} \bigg).
    \end{equation}

    Given two lists $\ii$ and $\ss$ of positive integers,
    we consider the three following graphs:
    \begin{itemize}
        \item $H_1$ has vertex set $[\ell]$ and has an edge between
            $j$ and $k$ if $|i_j-i_k| \leq 2$ {\em and} $|s_j -s_k| \leq 2$;
        \item $H_2$ has vertex set $[\ell]$ and has an edge between
            $j$ and $k$ if
            \[ \{i_j,i_j \pm 1,s_j,s_j \pm 1\} \cap \{i_k,i_k \pm 1,s_k,s_k \pm 1\}
            \neq \emptyset.\]
        \item $H$ has vertex set $[\ell] \sqcup [\ell]$ and has an edge
            between $j$ and $k$ (resp. $j$ and $\bar{k}$, $\bar{j}$ and $\bar{k}$)
            if $|i_j-i_k| \leq 2$ (resp. $|i_j-s_k| \leq 2$, $|s_j -s_k| \leq 2$)
    \end{itemize}
    We will use Theorem \ref{ThMain} to give a bound for 
    \[\bigg|\kappa \big( B^{(N)}_{i_1,s_1} B^{(N)}_{i_1+1,s_1+\epsilon_1}, \cdots,
        B^{(N)}_{i_\ell,s_\ell} B^{(N)}_{i_\ell+1,s_\ell+\epsilon_\ell} \big)\bigg| \]
    Clearly, the number $M(\ii,\ss)$ of different pairs in the set 
    \[ \{(i_j,s_j); 1 \leq j \leq \ell\} \cup 
        \{(i_j+1,s_j+\epsilon_j); 1 \leq j \leq \ell\} \]
    is at least equal to $2\#(\Conn(H_1)) \geq \#(\Conn(H_1))+1$.
    Besides, in this case, the graph $G'_2$ introduced in Section \ref{SubsectMainResult}
    has the same vertex set as $H_2$ and fewer edges.
    Hence it has more connected components.
    Therefore, Theorem \ref{ThMain} implies (step 2):
    \[\bigg|\kappa \big( B^{(N)}_{i_1,s_1} B^{(N)}_{i_1+1,s_1+\epsilon_1}, \cdots,
        B^{(N)}_{i_\ell,s_\ell} B^{(N)}_{i_\ell+1,s_\ell+\epsilon_\ell} \big)\bigg|
     \leq C_{2\ell} N^{-\#(\Conn(H_1))-\#(\Conn(H_2))}.\]
     But, using the terminology of Section \ref{SubsectFibers2},
     the graphs $H_1$ and $H_2$ are the strong and usual quotients of $H$.
     Therefore, by Lemma \ref{LemGraphFibers2}, one has:
    \begin{equation}\label{EqCCAdj}
        \#(\Conn(H)) \leq \#(\Conn(H_1))+\#(\Conn(H_2)).
    \end{equation}

    Besides, Lemma \ref{LemNumberListsGen} implies
    the number of lists $\ii$ and $\ss$ with entries in $[N]$
    such that $H$ has exactly
    $t$ connected components is bounded from above by $C''_{2\ell,\DD} N^t$
    for $\DD$ well-chosen (step 3).
    In particular the constant
    $C''_{2\ell,\DD}$ does not depend on $N$.
    Therefore, the total contribution of these lists to equation \eqref{EqDevCumAdj}
    is bounded from above by $C_{2\ell} N^{-t} \cdot C''_{2\ell,\DD} N^t
    = C_{2\ell} \cdot C''_{2\ell,\DD}$.

    Finally,
    \[\kappa_\ell(A^{(N)})= O(1).\]
    Moreover, only lists such that $M(\ii,\ss)=2$ and $\#(\Conn(H_1))=1$
    contribute to the term of order $1$.
    But this implies that the lists
    $\ii$, $\ss$ and $\bm{\varepsilon}$ are constant.
    In other words,
    \[\kappa_\ell(A^{(N)})
    = \sum_{i,s \geq 1 \atop \epsilon = \pm 1} \kappa_\ell(B^{(N)}_{i,s} B^{(N)}_{i+1,s+\epsilon})
    +O(N^{-1}).\]
    The $2(N-1)^2$ variables $B^{(N)}_{i,s} B^{(N)}_{i+1,s+\epsilon}$
    are Bernoulli variables, whose parameters are given by:
    \begin{itemize}
        \item if $s=i \in [N-1]$ and $\epsilon=1$, then the parameter is 
            $\frac{\theta^2}{(N+\theta-1)(N+\theta-2)}$ ($N-1$ cases);
        \item  if $s=i;\epsilon=-1$ (here $2\leq i \leq N-1$)
            or $s=i+1;\epsilon=-1$ (here $1\leq i \leq N-1$)
            or $s=i+2;\epsilon=-1$ (here $1\leq i \leq N-2$),
            then the parameter is $\frac{\theta}{(N+\theta-1)(N+\theta-2)}$
            ($3N-5$ cases);
        \item otherwise, the parameter is $\frac{1}{(N+\theta-1)(N+\theta-2)}$.
    \end{itemize}
    Recall that the cumulants of a sequence of Bernoulli variables $X^{(N)}$
    with parameters $(p_N)_{N\geq 1}$ with $p_N \to 0$ are asymptotically given
    by $\kappa_\ell(X^{(N)})=p_N+ O (p_N^2)$.
    Hence,
    \[\kappa_\ell(A^N)= 2(N-1)^2 \frac{1}{(N+\theta-1)(N+\theta-2)} + 
    O\big(N^{-1} \big) = 2 + O\big(N^{-1} \big).\]
    Finally, the cumulants of $A^{(N)}$ converges towards those of a Poisson variable
    with parameter $2$, which implies the convergence of $A^{(N)}$ in distribution.
\end{proof}

\subsection{Dashed patterns}\label{SubsectDashedPatterns}

In this section, we prove Theorem \ref{ThmDashedPatterns}, which describes,
for any given dashed pattern $(\tau,X)$,
the asymptotic behavior of the sequence $(O_{\tau,X}^{(N)})_{N\geq 1}$
of random variables.

\begin{proof}
    As in the previous examples, we write the quantity we want to study in terms
    of the variables $B^{(N)}_{i,s}$. Here,
    \[ O^{(N)}_{\tau,X} = \sum_{i_1 < \dots < i_p \atop \forall x \in X, i_{x+1}=i_x+1}
    \sum_{s_1,\dots,s_p \atop s_{\tau^{-1}(1)} < \dots < s_{\tau^{-1}(p)} }
    B^{(N)}_{i_1,s_1} \dots B^{(N)}_{i_p,s_p}. \]
    Expanding its cumulants by multilinearity, we get (step 1)
    \begin{equation}\label{EqDevCumOcc}
        \kappa_\ell ( O^{(N)}_{\tau,X} ) = \sum_{(i_j^r)} \sum_{(s_j^r)} \kappa \bigg(
    B^{(N)}_{i^1_1,s^1_1} \dots B^{(N)}_{i^1_p,s^1_p}, \dots,
    B^{(N)}_{i^\ell_1,s^\ell_1} \dots B^{(N)}_{i^\ell_p,s^\ell_p} \bigg).
    \end{equation}
    The first (resp. second) summation index is the set of matrices
    $(i_j^r)$ (resp. $(s_j^r)$) with $(j,r) \in [p] \times [\ell]$ such that:
    \begin{itemize}
        \item $\forall r, \ i^r_1 < \dots < i^r_p$ 
            (resp. $s^r_{\tau^{-1}(1)} < \dots < s^r_{\tau^{-1}(p)}$);
        \item $\forall r, \forall x \in X, \ i^r_{x+1}=i^r_x+1$
            (resp. no extra condition on the $s$'s).
    \end{itemize}
    Given such lists $\ii$ and $\ss$, we consider the four following graphs:
    \begin{itemize}
        \item $H_1$ has vertex set $[p] \times [\ell]$ and has an edge between
            $(j,r)$ and $(k,t)$ if $|i^r_j-i^t_k| \leq 1$ {\em and} $s^r_j = s^t_k$;
        \item $H_2$ has vertex set $[p] \times [\ell]$ and has an edge between
            $(j,r)$ and $(k,t)$ if
            \[ \{i^r_j,i^r_j+1,s^r_j\} \cap \{i^t_k,i^t_k+1,s^t_k\}
            \neq \emptyset.\]
        \item $H$ has vertex set $([p] \times [\ell]) \sqcup ([p] \times [\ell])$
            and has an edge
            between $(j,r)$ and $(k,t)$ (resp. $(j,r)$ and $\overline{(k,t)}$;
            $\overline{(j,r)}$ and $\overline{(k,t)}$)
            if $|i^r_j-i^t_k| \leq 1$ (resp. $s^t_k - i_j^r=0$ or $1$; $s^r_j=s^t_k$).
        \item $H'_2$ has vertex set $[\ell]$ and has an edge between $r$ and $t$ if
            \[ \left(\bigcup_{1\leq j \leq p} \{i_j^r, i^r_j+1, s_j^r\} \right)\cap
             \left( \bigcup_{1\leq k \leq p} \{i^t_k,i^t_k+1,s^t_k\} \right) \neq 
             \emptyset.\]
    \end{itemize}
    The graphs $H_1$ and $H_2$ are respectively the strong and usual quotients
    of $H$, as defined in Section \ref{SectGraph}.
    Therefore, one has, by Lemma \ref{LemGraphFibers2}:
    \[\#(\Conn(H)) \leq \#(\Conn(H_1)) + \#(\Conn(H_2)).\]
    But one can further contract $H_2$ by the map $f:[p] \times [\ell] \to [\ell]$
    defined by $f(j,r)=r$ and we obtain $H'_2$.
    With the notation of Section \ref{SectGraph}, it implies:
    \[\#(\Conn(H_2)) \leq \#(\Conn(H'_2)) + 
    \sum_{r=1}^\ell \left[\#\big(\Conn\big(
    H_2\big[ [p] \times \{r\}\big] \big)\big)-1 \right].\]
    But each induced graph $H_2[ [p] \times \{r\}]$ (for $1 \leq r \leq \ell$)
    contains at least an edge
    between $(x,r)$ and $(x+1,r)$ for each $x \in X$ (because we assumed 
    $i^r_{x+1}=i^r_x+1$).
    Thus it has at most $p-q$ connected components.
    Finally,
    \begin{equation}\label{EqTechConn}
        \#(\Conn(H)) \leq \#(\Conn(H_1)) + \#(\Conn(H'_2)) + (p-q-1)\ell.
    \end{equation}
    
    Let us apply the main lemma (Theorem \ref{ThMain}) to obtain a bound for
         \[ \bigg|\kappa \bigg(
        B^{(N)}_{i^1_1,s^1_1} \dots B^{(N)}_{i^1_p,s^1_p}, \dots,
            B^{(N)}_{i^\ell_1,s^\ell_1} \dots B^{(N)}_{i^\ell_p,s^\ell_p} \bigg) \bigg|.\]
      In this case, the number of different pairs in the indices of the Bernoulli
      variables is at least the number of connected components of $H_1$.
      Besides, the graph $G'_2$ introduced in Section \ref{SubsectMainResult}
      has the same vertex set, but fewer edges than $H'_2$.
      Hence, it has more connected components and we obtain:
               \[ \bigg|\kappa \bigg(
        B^{(N)}_{i^1_1,s^1_1} \dots B^{(N)}_{i^1_p,s^1_p}, \dots,
            B^{(N)}_{i^\ell_1,s^\ell_1} \dots B^{(N)}_{i^\ell_p,s^\ell_p} \bigg) \bigg|
            \leq C_{p \ell} N^{-\#(\Conn(H_1))-\#(\Conn(H'_2))+1}.\]
            Using inequality~\eqref{EqTechConn} above, this can be rewritten as (step 2)
    \[ \bigg|\kappa \bigg(                               
            B^{(N)}_{i^1_1,s^1_1} \dots B^{(N)}_{i^1_p,s^1_p}, \dots, 
                        B^{(N)}_{i^\ell_1,s^\ell_1} \dots B^{(N)}_{i^\ell_p,s^\ell_p} \bigg)\bigg|
                        \leq C_{p \ell} N^{-\#(\Conn(H))+ (p-q-1)\ell+1}.\]
    As in the previous section, Lemma \ref{LemNumberListsGen}
    asserts that the number of pairs of lists $((i^r_j),(s_j^r))$
    such that $\#(\Conn(H))=t$ is smaller than $C''_{2 p \ell,\DD} N^t$
    for a well chosen $\DD$ (step 3).
    Hence their total contribution to Equation~\eqref{EqDevCumOcc} is bounded from above by
    $C_{p \ell} C''_{p \ell,\DD} N^{(p-q-1)\ell+1}$.
    Finally, one has:
    \begin{equation}\label{EqBoundCumOcc}
        \kappa_\ell(O^{(N)}_{(X,\tau)})=O(N^{(p-q-1)\ell+1}),
    \end{equation}
    or equivalently $\kappa_\ell(Z^{(N)}_{(X,\tau)}) = O(N^{-\ell/2+1})$.
    As in Section \ref{SubsectConv}, the theorem follows from this bound and
    from the limits of the normalized expectation and variance.

    For the expectation, we have to consider the case $\ell=1$.
    In this case, one has $\#(\Conn(H_1))=p$ and $\#(\Conn(H'_2))=1$.
    Therefore, if we want an equality in Equation~\eqref{EqTechConn},
    we need $\#(\Conn(H))=2p-q$, which implies that all
    entries in the lists $\ii$ and $\ss$ are distinct.
    For these lists, one has (Lemma \ref{LemDistProdB})
    \[\kappa ( B^{(N)}_{i^1_1,s^1_1} \dots B^{(N)}_{i^1_p,s^1_p})=
    \esper ( B^{(N)}_{i^1_1,s^1_1} \dots B^{(N)}_{i^1_p,s^1_p})=
    \frac{1}{(N+\theta-1)_p}.\]
    But the number of lists with distinct entries in the index set of
    equation \eqref{EqDevCumOcc} is asymptotically $\frac{N^{2p-q}}{p!(p-q)!}$.
    Finally,
    \[ \lim_{N \to \infty} \frac{1}{N^{p-q}} \esper(O^{(N)}_{(X,\tau)})
    =\frac{1}{p!(p-q)!}.\]

    It remains to prove that the renormalized variance 
    $N^{-2(p-q)+1} \kappa_2(O^{(N)}_{(X,\tau)})$ has a limit $V_{\tau,X} \geq 0$,
    when $N$ tends to infinity.
    But this follows from the bound~\eqref{EqBoundCumOcc} and the fact that
    any $\kappa_\ell(O^{(N)}_{(X,\tau)})$
    is a rational function in $N$.
    Let us explain the latter fact.

    Recall that $\kappa_\ell(O^{(N)}_{(X,\tau)})$ is given by equation~\eqref{EqDevCumOcc}.
    We can split the sum depending on the graph $H$ associated to the matrices $\ii$ and $\ss$
    and on the actual value $\delta_e(\ii,\ss)$ of $i_j^r-i_k^t$ 
    (or $s^t_k - i_j^r$ and $s^t_k - s_j^r$ respectively) for each edge $e$ of $H$.
    Then the fact that $\kappa_\ell(O^{(N)}_{(X,\tau)})$ is a rational function
    is an immediate consequence of the following points:
    \begin{itemize}
    \item the numbers of graphs $H$ and of possible values
        for the differences $\delta_e(\ii,\ss)$ 
     (for $e \in E_{H}$) are finite;
     \item the cumulant $\kappa \big(
        B^{(N)}_{i^1_1,s^1_1} \dots B^{(N)}_{i^1_p,s^1_p}, \dots,
            B^{(N)}_{i^\ell_1,s^\ell_1} \dots B^{(N)}_{i^\ell_p,s^\ell_p} \big)$
            is a rational function in $N$ which depends only on the graph $H$
             and values of $\delta_e(\ii,\ss)$ (for $e \in E_{H}$);
     \item the number of matrices $\ii$ and $\ss$ corresponding to a given graph $G$ and given values
     $\delta_e(\ii,\ss)$ is a polynomial in $N$. \qedhere
     \end{itemize}        
\end{proof}


\subsection{Conclusion: local statistics}\label{SubsectLocalStat}
Recently, several authors have further generalized the notion of dashed patterns
into the notion of {\em bivincular} patterns 
\cite[Section 2]{BousquetEtAlGeneralizedPattern}.
The idea is roughly that, in an occurrence of a bivincular pattern,
one can ask that some values are consecutive (and not only some places
as in dashed patterns).
This new notion is very natural as occurrences of bivincular patterns in the inverse
of a permutation correspond to occurrences of bivincular patterns in the
permutation itself (which is not true for dashed patterns).

It would be interesting to give a general theorem
on the asymptotic behavior of the number of occurrences
of a given bivincular pattern.
This seems to be a hard problem as many different behavior can occur:
\begin{itemize}
    \item The number of adjacencies is the sum of the number of occurrences
        of two different bivincular patterns and converge towards a Poisson
        distribution.
    \item The dashed patterns are special cases of bivincular patterns.
        As we have seen in the previous section, their number of occurrences
        converges, after normalization, towards a Gaussian law
        (at least for patterns of size smaller than $9$,
        the general case relies on Conjecture~\ref{ConjVariance}).
        Other bivincular patterns exhibit the same behavior, for example
        the one considered in \cite{BousquetEtAlGeneralizedPattern}.
    \item Other behaviors can occur: for example,
        it is easy to see that the number of occurrences of
        the pattern $(123,\{1\},\{1\})$ (we use  
        the notations of \cite{BousquetEtAlGeneralizedPattern}),
        has an expectation of order $n$, but a
        probability of being $0$ with a positive lower bound.
\end{itemize}
Unfortunately, we have not been able to give a general statement.
Let us however emphasize the fact that
our approach unifies the first two cases.

More generally, our approach seems suitable to study 
what could be called a {\em local statistic}.
Fix a integer $p\geq 1$ and a set $S$ of constraints:
a constraint is an equality or inequality (large or strict) whose members
are of the form $i_{j}+d$ or $s_j+d$ where $j$ belongs to $[p]$ and $d$ is some integer.
Then, for a permutation $\sigma$ of $S_N$, we define $O^{(N)}_{p,S}(\sigma)$ as the 
number of lists $i_1,\dots,i_p$ and $s_1,\dots,s_p$ satisfying the constraints
in $S$ and such that $\sigma(i_j)=s_j$ for all $j$ in $[p]$.
For instance, the number of $d$-descents
studied in \cite{BonaNormalityDDescents}
is a local statistic.

We call any linear combination of statistics $O^{(N)}_{p,S}$ a local statistic.
The number of occurrences of a bivincular patterns, but also the number of
exceedances or of cycles of a given length $p$, are examples of local statistics.
The method presented in this article is suitable for the asymptotic study of
joint vectors of local statistics.
We have failed to find a general statement, but we are convinced that our approach
can be adapted to many more examples than the ones studied in this article.

However, the method does not seem appropriate to {\em global} statistics, such as
the total number of cycles of the permutation or the length of the longest cycle.


\ACKNO{
The author would like to thank Jean-Christophe Aval, Adrien Boussicault and
Philippe Nadeau for sharing their interest in permutation tableaux,
Mireille Bousquet-Mélou and Mathilde Bouvel for pointing out some references
and, last but not least, Jean-François Marckert for a lot of explanations
and references on convergence of random functions and Gaussian processes
(in particular, his notes from lectures given in Graz, available online,
have been very useful for me).

Besides, the author wants also to thank some anonymous referees
for a lot of lingual, mathematical and organizational 
corrections or suggestions
that improved the presentation of the paper.
}


\end{document}